\documentclass[11pt]{article}%
\usepackage{amsmath}
\usepackage{amssymb}
\usepackage{amsfonts}
\usepackage{graphicx}
\usepackage{enumitem}
\usepackage{hyperref}
\usepackage{color}%
\usepackage[american,fulldiode]{circuitikz} 

\providecommand{\U}[1]{\protect\rule{.1in}{.1in}}
\allowdisplaybreaks
\newtheorem{theorem}{Theorem}[section]

\newtheorem{proposition}[theorem]{Proposition}
\newtheorem{corollary}[theorem]{Corollary}
\newtheorem{lemma}[theorem]{Lemma}
\newtheorem{conjecture}[theorem]{Conjecture}
\newtheorem{remark}[theorem]{Remark}

\newtheorem{example}[theorem]{Example}

\newtheorem{algorithm}[theorem]{Algorithm}
\newenvironment{proof}{\noindent{\em Proof:}}{$\Box$~\\}

\newcommand{\sign}{{\rm sign~}}
\newcommand{\bR}{\mathbb R}
\begin{document}

\title{Locating and counting equilibria of the Kuramoto model \\ with rank one coupling}
\author{Owen Coss\thanks{Department of Mathematics, North Carolina State University (otcoss@ncsu.edu, \url{www.math.ncsu.edu/\~otcoss}). }
\and Jonathan D. Hauenstein\thanks{Department of Applied and Computational
Mathematics and Statistics, University of Notre Dame (hauenstein@nd.edu,
\url{www.nd.edu/\~jhauenst}). This author was partially supported by NSF grant
ACI-1460032, Sloan Research Fellowship BR2014-110 TR14, U.S. Army Research Office grant W911NF-15-1-0219 under the Young Investigator Program, and Office of Naval Research grant
N00014-16-1-2722.}
\and Hoon Hong\thanks{Department of Mathematics, North Carolina State University
(hong@ncsu.edu, \url{www.math.ncsu.edu/\~hong}). This author was partially supported by NSF grant 1319632.}
\and Daniel K. Molzahn\thanks{Energy Systems Division, Argonne National Laboratory
(dmolzahn@anl.gov).} }
\date{\today}
\maketitle

\begin{abstract}
\noindent The Kuramoto model describes synchronization behavior among coupled oscillators and enjoys successful application in a wide variety of fields. Many of these applications seek 
phase-coherent solutions, i.e., equilibria of the model. Historically, research has focused on situations where the number of oscillators, $n$, is extremely large and can be treated as being infinite. More recently, however, applications have arisen in areas such as electrical engineering with more modest values of $n$. For these, the equilibria can be located by finding the real solutions of a system of polynomial equations utilizing techniques from algebraic geometry. However, typical methods for solving such systems locate all complex solutions even though only the real solutions give equilibria.

In this paper, we present an algorithm to locate only the real solutions 
of the model, thereby shortening computation time
by several orders of magnitude in certain situations. This is accomplished by choosing specific equilibria representatives and the consequent algebraic decoupling
of the system. The correctness of the algorithm (that it finds only and all the equilibria) is proved rigorously. Additionally, the algorithm can be implemented using interval methods so that the equilibria can be approximated up to any given precision without significantly more computational effort. We also compare this solving approach to other computational algebraic geometric methods. 

Furthermore, analyzing this approach allows us to prove, asymptotically, that the maximum number of 
equilibria grows at the same rate as the number of complex solutions of a corresponding polynomial
system. Finally, we conjecture an upper bound on the maximum number of equilibria for any number of oscillators which generalizes the known cases and is obtained on a range of explicitly provided natural frequencies.

\medskip


\medskip


\end{abstract}


\section{Introduction}

\label{sec:Intro} 

Oscillatory dynamics characterize many important systems. 
For such systems, it is important
to understand the synchronization behavior of coupled oscillators, especially when conducting stability assessments. 
Synchronization behavior is characterized by the equilibria of the associated dynamic model.
This paper is concerned with locating and counting  equilibria of a certain generalization of the Kuramoto model~\cite{kuramoto1975}, which we call a {\em rank-one coupled}
 Kuramoto model.
 
\paragraph{Kuramoto model:}
The standard Kuramoto model for $n\geq 2$ oscillators
has all-to-all and uniform coupling among the
oscillators.  It is formulated as
the following system of coupled first-order ordinary differential equations:
\begin{equation}
\frac{d\theta_{\nu}}{dt}=\omega_{\nu}-\frac{K}{n}\sum_{\mu=1}^{n}\sin(\theta_{\nu}%
-\theta_{\mu}),\hbox{~~~~~~~for~}\nu=1,\dots,n \label{eq:Kuramoto}%
\end{equation}
where $K>0$ is the uniform coupling strength,
and each parameter $\omega_{\nu}$ and variable $\theta_{\nu}$ 
denote the natural frequency and
phase angle of the $\nu^{\rm th}$ oscillator,
respectively.
There is a large body of literature for the Kuramoto model~\eqref{eq:Kuramoto} and its
many variants, e.g., non-uniform coupling among oscillators and allowance for
second-order dynamics.  
The wide variety of applications of the Kuramoto model in modeling oscillatory behavior include
electrical engineering~\cite{dorfler_bullo2012, dorfler2013, wiesenfeld1998},
biology~\cite{sompolinsky1990}, and
chemistry~\cite{bar-eli1985, kuramoto1984, neu1980}.
See~\cite{acebron2005, dorfler2013, strogatz2000} and the references therein 
for a more detailed survey of the relevant literature
and extensive applications.

\paragraph{Rank-one coupled Kuramoto model:}
In this paper, we consider a slight generalization
with a non-uniform coupling between the oscillators
described by a symmetric rank-one matrix.  
In particular, for $k = (k_1,\dots,k_n)\in\bR_{>0}^n$, 
the $\nu^{\rm th}$ and $\mu^{\rm th}$ oscillators
are coupled with strength $k_\nu k_\mu$ yielding the
model
\begin{equation}
\frac{d\theta_{\nu}}{dt}=\omega_{\nu}-\frac{1}{n}\sum_{\mu=1}^{n} k_\nu k_\mu\sin(\theta_{\nu}%
-\theta_{\mu}),\hbox{~~~~~~~for~}\nu=1,\dots,n. \label{eq:model_equiv}%
\end{equation}
The standard Kuramoto case~\eqref{eq:Kuramoto} corresponds with $k = (\sqrt{K},\dots,\sqrt{K})$.
We are concerned with the equilibria  of the rank-one coupled Kuramoto model
\eqref{eq:model_equiv}, which are the real solutions to the
system of nonlinear equations resulting from setting $\dfrac{d\theta_{\nu}}{dt}$ equal to $0$ in \eqref{eq:model_equiv}, namely 
\begin{equation}
\omega_{\nu}=\frac{1}{n}\sum_{\mu=1}^{n}k_\nu k_\mu \sin(\theta_{\nu}-\theta_{\mu}),\hbox{~~~~~~~for~}\nu
=1,\dots,n. \label{eq:Kuramoto_equilibrium}%
\end{equation}
That is, we aim to compute the values of the variables $\theta_1,\dots,\theta_n$ such that \eqref{eq:Kuramoto_equilibrium}
holds for given values of the parameters $n$, $k_1,\dots,k_n$,
and $\omega_1,\dots,\omega_n$. This generalization was originally motivated by applications where the coupling is non-uniform, such as in a power flow model~\cite{dorfler_bullo2012, dorfler2013, wiesenfeld1998}
in which the coupling matrix could be of arbitrary rank.
However, as demonstrated in Ex.~\ref{ex:PowerFlow},
with a lossless power system and uniform line
susceptances, the equilibria of the power
flow equations correspond to the equilibria
of the rank-one coupled Kuramoto model~\eqref{eq:Kuramoto_equilibrium}.
Hence, \eqref{eq:Kuramoto_equilibrium}
can be viewed as an initial generalization 
(rank one) toward the full generalization (arbitrary rank).

We address two natural problems: locating all 
equilibria and counting them.

\paragraph{Locating all equilibria:}
In~\cite{MMT16, mehta2015algebraic, salam1989}, 
homotopy continuation
and numerical algebraic geometry~\cite{BHSW06, sommese2005}
were applied to the standard Kuramoto
model and various non-uniform coupling 
generalizations by converting 
the corresponding system describing the
equilibria into a polynomial system.
For example, 
with $s_{\nu}=\sin(\theta_{\nu})$ and $c_{\nu}=\cos(\theta_{\nu})$,
\eqref{eq:Kuramoto_equilibrium} 
corresponds to the polynomial~system 
\begin{align}
\label{eq:Kuramoto_poly}%
\omega_{\nu}  
 =\frac{1}{n}\sum_{\mu=1}^{n} k_\nu k_\mu
\left(  s_{\nu}c_{\mu}-s_{\mu}c_{\nu}\right),
& & & 1  =c_{\nu}^{2}+s_{\nu}^{2}, 
& \;\;\;\;\;\;\hbox{~for~}\nu=1,\dots,n.
\end{align}

Even though all complex solutions were computed,
only the real solutions are physically meaningful, 
i.e., correspond to equilibria,
so that a post-processing step is necessary to filter out the non-real solutions.  
In other words, homotopy continuation
expends computational effort to compute all
complex solutions when only the real solutions are relevant. 
Using parallel computing techniques, such a method
has been applied to problems with $n\leq18$~\cite{mehta2015algebraic}. 

In~\cite{thorp1993}, a specialized continuation method 
was proposed which computes only equilibria
so that the computational cost 
scales with the number of
real solutions of the corresponding polynomial
system rather than the number of complex
solutions. Moreover, this continuation method is applicable to a more general class of problems (the power flow equations) which include~\eqref{eq:Kuramoto_equilibrium} as a special case. However, the robustness proof 
showing that it locates all equilibria for this more general class of problems
was shown to be flawed~\cite{chen2011} 
with a counterexample presented in~\cite{counterexample2013}. 
In \cite{lesieutre_wu_allerton2015}, a modification of the
method based on an elliptical reformulation of 
equations was shown to have improved
robustness. There currently does not exist a
robustness proof for this modification or a known
counterexample, so the capabilities of this method remain to be fully characterized.

In summary, despite significant progress,      
the aforementioned approaches either quickly become
intractable as $n$ increases 
or are not proven to find all equilibria.
One of the main contributions of this paper is to provide  a new algorithm that can handle much larger 
values of $n$ which is also rigorously proved 
to find all equilibria.  For instance, in Section~\ref{sec:Performance}, we demonstrate
our approach on an example with $n = 60$
which computes all equilibria in under a second.

\paragraph{Counting equilibria:}
The second problem is to determine the maximum number of equilibria (up to trivial shifts -- see Section~\ref{sec:Reduction}).
Existing upper bounds on the number of equilibria are based on bounds for the number of
complex solutions to~\eqref{eq:Kuramoto_poly}. 
In~\cite{baillieul1982}, an upper bound on the number
of equilibria of the Kuramoto model with an
arbitrary coupling matrix $\kappa\in\bR^{n\times n}$,
i.e., the equilibria satisfy
\begin{equation}
\omega_{\nu}=\frac{1}{n}\sum_{\mu=1}^{n}\kappa_{\nu\mu} \sin(\theta_{\nu}-\theta_{\mu}),\hbox{~~~~~~~for~}\nu
=1,\dots,n, \label{eq:Kuramoto_equilibrium2}%
\end{equation}
is $\binom{2n-2}{n-1}$.
This bound is sharp for $n=2$ and $n=3$. 
It is an open
question (first posed in~\cite[Question~5.1]{baillieul1982}) whether the
upper bound of $\binom{2n-2}{n-1}$ can be achieved for $n\geq4$. 

Other research~\cite{CM15, CMN16, MMN16} 
has produced tighter upper bounds on the number
of complex solutions to~\eqref{eq:Kuramoto_equilibrium2} 
when the oscillators are not completely 
connected, i.e., ``topologically dependent''~bounds.

The number of equilibria for the standard Kuramoto model
has been studied for small values of $n$.  
In the standard Kuramoto setting,
i.e., $k = (\sqrt{K},\dots,\sqrt{K})$, there are at most $2$ equilibria satisfying~\eqref{eq:Kuramoto_equilibrium} when $n = 2$.  
For $n = 3$ and $n = 4$, elimination theory
was used in \cite{XKL16} to produce
a degree six and degree fourteen univariate polynomial, respectively, yielding bounds of at most $6$ and $14$
equilibria.  Morse Theory was used to derive similar results 
for the $n=3$ and $n=4$ cases in~\cite{baillieul1982}.
These aforementioned bounds are tight for
$n = 2$ and $n = 3$, but it is currently
unknown whether the upper bound of $14$ 
can be achieved for $n=4$. 
The authors of~\cite{XKL16} find a maximum of $10$ equilibria in the $n = 4$ case, which is smaller than the upper bound of~$14$.  Since this maximum was obtained via a computational experiment which gridded the parameter space, they conjecture that~$10$
is indeed the maximum number of equilibria
when $n = 4$.

In summary, despite significant progress,
there remains several open questions
regarding the number of equilibria
to the rank-one coupled Kuramoto model.
First, for the polynomial system~\eqref{eq:Kuramoto_poly}, 
the generic root count, which is the 
number of solutions for generic
values of the parameters,
is unknown.  Clearly, this is bounded
above by $\binom{2n-2}{n-1}$
which is the generic root count for the corresponding
polynomial system 
in the arbitrary coupling case \eqref{eq:Kuramoto_equilibrium2}.
Moreover, the quality of the relationship
between the maximum number of equilibria 
and the generic root count has not been explored.
Three contributions of this paper
are to provide such a generic root count for~\eqref{eq:Kuramoto_poly}, count the number
of equilibria in particular cases, 
and use these cases to analyze the asymptotic
behavior of the ratio between the maximum number of equilibria
and the generic root count for \eqref{eq:Kuramoto_poly}.

\paragraph{Approach:}
This paper locates and counts
equilibria for arbitrary $n$ by reformulating~\eqref{eq:Kuramoto_equilibrium} 
into a family of decoupled univariate radical equations. (This
reformulation is similar in spirit but different than the approach
in~\cite{XKL16}. Further, the proposed reformulation is not limited to
$n=2,3,4$.) 
This reformulation enables the development of 
both new theoretical results and computational tools. 
Our solving algorithm exploits 
this reformulation together with new results
regarding cases where equilibria cannot exist.
Computational experiments demonstrate that this
algorithm can be several orders of magnitude faster 
than the more general
computational algebraic geometric 
methods~\cite{BHSW06, M2, MMT16, mehta2015algebraic, salam1989, sommese2005}
and elliptical continuation~\cite{lesieutre_wu_allerton2015, thorp1993} algorithms
when applied to~\eqref{eq:Kuramoto_poly}. 

This reformulation allows us to count
the number of equilibria 
for \eqref{eq:Kuramoto_equilibrium} 
where the parameters are carefully chosen
to have many equilibria.  
These results extend a conjecture from~\cite{XKL16}
that the maximum number of equilibria
for the standard Kuramoto model when $n = 4$ is $10$.
Moreover, the particular cases 
allow us to show that the 
maximum number of equilibria 
and the generic root count 
for~\eqref{eq:Kuramoto_poly} have
the same asymptotic scaling.
This suggests that algorithms which only compute
equilibria to~\eqref{eq:Kuramoto_equilibrium} will,
in the worst-case, computationally scale similar to algorithms that 
compute all the complex solutions~to~\eqref{eq:Kuramoto_poly}. However, it may be the case that algorithms which compute
only the real solutions to~\eqref{eq:Kuramoto_poly} have significant
computational advantages for many practical problems. For instance, typical operating conditions 
of power flow problems
are expected to have few equilibria
relative to the number of complex solutions~\cite{salam1989}.

The rest of the paper is organized as follows. 
Section~\ref{sec:Reduction} 
presents the decoupling approach.
Section~\ref{sec:Algorithm} describes an algorithm that uses 
this reformulation to compute all equilibria 
satisfying~\eqref{eq:Kuramoto_equilibrium}
and compares the computational performance of this algorithm with other methods.
Section~\ref{sec:UpperBound} 
counts the number of equilibria in particular
cases and compares the maximum number of 
equilibria with the generic root count of \eqref{eq:Kuramoto_poly}.
A short conclusion is provided in Section~\ref{sec:Conclusion}.

\section{Decoupled reformulation}
\label{sec:Reduction}
One approach to solving a multivariate system of equations involves first decoupling the system. In this section, we take this approach and decouple the  system   \eqref{eq:Kuramoto_equilibrium}. A standard method for decoupling is to apply computational tools from elimination theory (i.e., multivariate resultants and
Gr{\"o}bner basis techniques~\cite{Buchberger1965, CE93, ZC2017, Cox2005, Cox2015, Emiris1999, Macaulay1902, Sturmfels1991}) to the  polynomial system~\eqref{eq:Kuramoto_poly}, thereby 
obtaining  several univariate polynomials, say $f$ and  $g_{1},\ldots,g_{2n}$, 
such that each solution of \eqref{eq:Kuramoto_poly} is the value of the polynomial system $g = (g_1,\dots,g_{2n})$ evaluated 
at a  root of $f$. 
However, the major drawback of this approach is that the obtained polynomials $f$ and $g_1,\dots,g_{2n}$ are of  very high degree (exponential in~$n$) with no naturally discernible structure. Thus solving with this method is very time-consuming, even for relatively small $n$.
In the following, we will instead use an alternate method to decouple the system \eqref{eq:Kuramoto_equilibrium} adapted from Kuramoto's approach \cite[\S 5.4]{kuramoto1975, kuramoto1984}. This method allows us to exploit  the inherent structure of the equations to obtain explicit radical expressions that are quickly solvable by standard univariate solvers.

Using \eqref{eq:model_equiv}, it is easy to see that
$$\sum_{\nu=1}^n \frac{d\theta_\nu}{dt} = \sum_{\nu = 1}^n \omega_\nu$$
so that equilibria can only exist when $\sum_{\nu = 1}^n \omega_\nu = 0$.  Therefore, we only need to 
consider solving such cases,
which we list as our first input condition~(IC):
\begin{enumerate}
\item[\textbf{{IC1}:}] $\omega_{1}+\cdots+\omega_{n}=0$.
\end{enumerate}

If $\theta = (\theta_1,\dots,\theta_n)$ is a solution of  \eqref{eq:Kuramoto_equilibrium},
then shifting all angles by $\phi$, i.e., $(\theta_1+\phi,\dots,\theta_n+\phi)$,
is also a solution.  Thus, we want to both 
compute and count equilibria modulo shift.  
One approach, e.g., as used in~\cite{mehta2015algebraic},
is to set one of the angles, say $\theta_n$, to be zero.
A second approach is to fix the ``weighted average angle'' 
which is the following output condition (OC):

\begin{enumerate}
\item[\textbf{{OC1}:}] $\sum_{\mu=1}^{n}k_{\mu}e^{i\theta_{\mu
}}\in \bR_{\geq0}$ 
\end{enumerate}
In paritcular, \textbf{OC1} is equivalent to selecting $\theta$
so that $\sum_{\mu=1}^n k_\mu \sin(\theta_\mu) = 0$
and $\sum_{\mu=1}^n k_\mu \cos(\theta_\mu) \geq 0$.
It is a natural  extension of the condition used 
by Kuramoto \cite[\S 5.4]{kuramoto1975, kuramoto1984}. 

When each $\omega_i = 0$, the following shows
that there can be infinitely many equilibria modulo~shift.

\begin{example}\label{ex:InfManyEq}
For $n = 4$ with $\omega_i = 0$ and $k_i = 1$, \eqref{eq:Kuramoto_equilibrium} is equivalent to
$$\begin{array}{ccccc}
\sin(\theta_1-\theta_2) + \sin(\theta_1-\theta_3) + \sin(\theta_1-\theta_4) &=& 
\sin(\theta_2-\theta_1) + \sin(\theta_2-\theta_3) + \sin(\theta_2-\theta_4) &=& \\ 
\sin(\theta_3-\theta_1) + \sin(\theta_3-\theta_2) + \sin(\theta_3-\theta_4) &=& 
\sin(\theta_4-\theta_1) + \sin(\theta_4-\theta_2) + \sin(\theta_4-\theta_3) &=& 0.
\end{array}$$
This system has infinitely many equilibria modulo shift
which can be seen, for example, by taking
$$\theta_3 = \theta_1 + \pi \hbox{~~and~~} \theta_4 = \theta_2 + \pi.$$
\end{example}

Since we aim to enumerate all
equilibria modulo shift, we will not consider the case when 
every~$\omega_i = 0$ and will leave this positive-dimensional case 
as a possible future research direction.
This forms our second input condition:
\begin{enumerate}
\item[\textbf{{IC2}:}] $\omega\neq\left(  0,\ldots,0\right) $
\end{enumerate}

Finally, we want to consider rank-one coupled
Kuramoto models that are {\em fully coupled}, i.e., $k_\nu > 0$
for $\nu = 1,\dots,n$, so that each oscillator is positively
impacted by every other oscillator.  This is a natural 
extension of the classical Kuramoto model \eqref{eq:Kuramoto} 
for which the uniform coupling strength~$K$ is positive.
With this assumption, we can, without loss of generality,
adjust the indexing to order the input parameters
based on $|\omega_\nu/k_\nu|$.  This forms our third input condition:
\begin{enumerate}
\item[\textbf{IC3}:] $k_{1},\ldots,k_{n}>0$ which
are ordered so that 
$\left| \dfrac{\omega_{1}}{k_{1}}\right| \leq\left|
\dfrac{\omega_{2}}{k_{2}}\right| \leq\cdots\leq\left| \dfrac{\omega_{n}}{k_{n}}\right|$
\end{enumerate}

With this setup, we are now ready to 
decouple the multivariate system of equations~\eqref{eq:Kuramoto_equilibrium}.

\begin{theorem}
[Decoupled Reformulation]\label{thm: reduction}
Suppose that $\omega\in\mathbb{R}^{n}$ and
$k\in\mathbb{R}_{>0}^{n}$ 
satisfy {\rm \textbf{IC1}, \textbf{IC2}, and~\textbf{IC3}}.
If~$\Theta_{\omega,k}$ is the set of all equilibria
described via {\rm \textbf{OC1}}
satisfying \eqref{eq:Kuramoto_equilibrium}, then
\begin{align*}
\Theta_{\omega,k}  &  =%
{\displaystyle\bigcup\limits_{\sigma\in\left\{  -1,+1\right\}  ^{n}}}
\Theta_{\omega,k,\sigma}  \hbox{~~~~~~~~~~where~} \\
\Theta_{\omega,k,\sigma}  &  =%
{\displaystyle\bigcup\limits_{R\in\mathcal{R}_{\omega,k,\sigma}}}
\left\{  \theta\in(-\pi,\pi]^{n}\ :\ \ 
\sin\theta_{\nu}=\frac{\omega_{\nu}}{k_{\nu}\sqrt{R}%
}\ \ \hbox{\rm and}\ \ \operatorname*{sign}\cos\theta_{\nu}=\sigma_{\nu}
\hbox{\rm~~~for~} \nu = 1,\dots,n\right\},
\\
\mathcal{R}_{\omega,k,\sigma}  &  =\left\{  R\in\mathbb{R}_{>0}\ :\ \ R=\frac
{1}{n}\sum_{\mu=1}^{n}\sigma_{\mu}\sqrt{k_{\mu}^{2}R-\omega_{\mu}^{2}%
}\right\}.
\end{align*}

\end{theorem}

\begin{proof}
For $U = (-\pi,\pi]^n$, we have 
\[
\Theta_{\omega,k}=\left\{  \theta\in U\ :\ \ \underset{\nu\in\left\{
1,\ldots,n\right\}  }{\forall}\omega_{\nu}=\frac{1}{n}\sum_{\mu=1}^{n}k_{\nu
}k_{\mu}\sin\left(  \theta_{\nu}-\theta_{\mu}\right)  \ \ \text{and
}\ \mathbf{OC1}\right\}.
\]
Since $\sin\eta=\operatorname{Im}e^{i\eta}$,
\textbf{IC3} and factoring yields
\[
\Theta_{\omega,k}=\left\{  \theta\in U\ :\ \ \underset{\nu\in\left\{
1,\ldots,n\right\}  }{\forall}\ \omega_{\nu}=\frac{1}{n}k_{\nu}%
\operatorname{Im}e^{i\theta_{\nu}}\sum_{\mu=1}^{n}k_{\mu}e^{-i\theta_{\mu}%
}\ \ \text{and }\ \mathbf{OC1}\right\}.
\]
From $e^{-i\alpha}=\overline{e^{i\alpha}}$ and \textbf{IC3}, we have
\[
\Theta_{\omega,k}=\left\{  \theta\in U\ :\ \ \underset{\nu\in\left\{
1,\ldots,n\right\}  }{\forall}\ \omega_{\nu}=\frac{1}{n}k_{\nu}%
\operatorname{Im}e^{i\theta_{\nu}}\overline{\sum_{\mu=1}^{n}k_{\mu}%
e^{i\theta_{\mu}}}\ \ \text{and }\ \mathbf{OC1}\right\}.
\]
Since \textbf{OC1} is equivalent to $\underset{r\in\mathbb{R}_{\geq0}%
}{\exists}\ r=\frac{1}{n}\sum_{\mu=1}^{n}k_{\mu}e^{i\theta_{\mu}}$, we have
\[
\Theta_{\omega,k}=\left\{  \theta\in U\ :\ \ \underset{r\in\mathbb{R}_{\geq
0}}{\exists}\ \text{ }r=\frac{1}{n}\sum_{\mu=1}^{n}k_{\mu}e^{i\theta_{\mu}%
}\ \ \text{and\ }\underset{\nu\in\left\{  1,\ldots,n\right\}  }{\forall
}\ \omega_{\nu}=rk_{\nu}\operatorname{Im}e^{i\theta_{\nu}}\right\}.
\]
If $r=0,$ then $\omega=\left(  0,\ldots,0\right),$ contradicting
\textbf{IC2}. Hence, $r\neq0$ so that
\[
\Theta_{\omega,k}=\left\{  \theta\in U\ :\ \ \underset{r\in\mathbb{R}%
_{>0}}{\exists}\ \text{ }r=\frac{1}{n}\sum_{\mu=1}^{n}k_{\mu}e^{i\theta_{\mu}%
}\ \ \text{and\ }\underset{\nu\in\left\{  1,\ldots,n\right\}  }{\forall
}\ \omega_{\nu}=rk_{\nu}\sin\theta_\nu\right\}.
\]
From \textbf{IC1}, we have $0=\sum_{\mu=1}^{n}\omega_{\mu}=\sum_{\mu=1}%
^{n}rk_{\mu}\sin\theta_{\mu}=r\sum_{\mu=1}^{n}k_{\mu}\sin\theta_{\mu}.$ Since
$r\neq0$, we know $\sum_{\mu=1}^{n}k_{\mu}\sin\theta_{\mu}=0$. Thus,
\[
\Theta_{\omega,k}=\left\{  \theta\in U\ :\ \ \underset{r\in\mathbb{R}%
_{>0}}{\exists}\ \text{\ }r=\frac{1}{n}\sum_{\mu=1}^{n}k_{\mu}\cos\theta_{\mu
}\ \ \text{and \ }\underset{\nu\in\left\{  1,\ldots,n\right\}  }{\forall
}\ \omega_{\nu}=rk_{\nu}\sin\theta_{\nu}\right\}.
\]
Since $\cos\alpha=\pm\sqrt{1-\sin^{2}\alpha}$
and $\sin\theta_{\mu} = \dfrac{\omega_{\mu}}{k_{\mu}r}$, we have
\begin{align*}
\Theta_{\omega,k}  &  =%
{\displaystyle\bigcup\limits_{\sigma\in\left\{  -1,+1\right\}  ^{n}}}
\Theta_{\omega,k,\sigma} \hbox{~~~~~~where}\\
\hbox{\footnotesize $\Theta_{\omega,k,\sigma}$} &=
\hbox{\footnotesize $\left\{  \theta\in U\ :\ \ \underset{r\in
\mathbb{R}_{>0}}{\exists}\ \text{\ }r=\displaystyle\frac{1}{n}\sum_{\mu=1}^{n}k_{\mu}%
\sigma_{\mu}\sqrt{1-\left(  \frac{\omega_{\mu}}{k_{\mu}r}\right)  ^{2}%
}\ \ \text{and \ }\underset{\nu\in\left\{  1,\ldots,n\right\}  }{\forall}%
\sin\theta_{\nu}=\frac{\omega_{\nu}}{k_{\nu}r}\ \ \text{and\ }%
\ \operatorname*{sign}\cos\theta_{\nu}=\sigma_{\nu}\right\}$.}
\end{align*}
Since $r,k_\mu > 0$, we can simplify to
\[
\hbox{\small $
\Theta_{\omega,k,\sigma}=\left\{  \theta\in U\ :\ \ \underset{r\in
\mathbb{R}_{>0}}{\exists}\ \text{\ }r^{2}=
\displaystyle\frac{1}{n}\sum_{\mu=1}^{n}\sigma_{\mu}\sqrt{k_{\mu}^2 r^{2}-\omega_{\mu}^{2}%
}\ \ \text{and \ }\underset{\nu\in\left\{  1,\ldots,n\right\}  }{\forall}%
\sin\theta_{\nu}=\frac{\omega_{\nu}}{k_{\nu}r}\ \text{and\ }%
\ \operatorname*{sign}\cos\theta_{\nu}=\sigma_{\nu}\right\}$}.
\]
Since $r > 0$, for $R = r^2>0$, we have $r = \sqrt{R}$ 
yielding the result.
\end{proof}

\begin{remark}
The proof of Theorem~\ref{thm: reduction} shows
that we could update {\rm\textbf{OC1}} to be
$$\sum_{\mu=1}^{n}k_{\mu}e^{i\theta_{\mu}}\in \bR_{>0}$$
which yields that a unique representative is computed
for each equilibria modulo shift.  
In particular, if $\psi=(\psi_1,\dots,\psi_n)$ is such
that $\sum_{\mu=1}^{n}k_{\mu}e^{i\psi_{\mu}}\neq0$,
then there is a unique $\phi\in(-\pi,\pi]$ such that
$$\sum_{\mu=1}^{n}k_{\mu}e^{i(\psi_{\mu}+\phi)}
= e^{i\phi} \sum_{\mu=1}^{n}k_{\mu}e^{i\theta_{\mu}}\in \bR_{>0}.$$
\end{remark}

\section{Locating the equilibria}\label{sec:Algorithm}

Theorem~\ref{thm: reduction} immediately
yields an algorithm for locating all equilibria
satisfying~\eqref{eq:Kuramoto_equilibrium}.
After some improvements, we compare
the resulting method with other approaches.

\subsection{Basic algorithm}

For each $\sigma\in\{-1,+1\}^n$, the first step
to utilize Theorem~\ref{thm: reduction} 
for locating all equilibria is to find
the positive roots of
\begin{equation}
\label{eq:Univar}
f_{\sigma}(R) = -R + \frac{1}{n}\sum_{\mu=1}^{n}\sigma_{\mu}\sqrt{k_{\mu}%
        ^{2}R-\omega_{\mu}^{2}}.
\end{equation}
The following algorithm depends upon a root finding method 
that returns the set of all positive roots of $f_\sigma$,
denoted $\mathbf{Solve}(f_\sigma,+)$.  Our implementation 
uses an interval Newton method \cite[Chap.~6]{hammer1995c++}, which allows each positive root of $f_\sigma$ to be approximated
up to any given precision. 
For each positive root $R$ of $f_\sigma$,
the second step from Theorem~\ref{thm: reduction}
is to compute the equilibria~via
\[
\sin \theta_\nu = \frac{\omega_\nu}{k_\nu \sqrt{R}} 
\hbox{~~~and~~~}
\operatorname*{sign}\cos \theta_\nu = \sigma_\nu
\hbox{~~~for~}\nu = 1,\dots,n.
\]
This is summarized in the following algorithm.

\begin{algorithm}[Basic]\label{alg:Basic}\ 

\begin{description}
[leftmargin=3em,style=nextline,itemsep=0.5em]

\item[In:] $\omega\in\mathbb{R}^{n}$ and $k\in\mathbb{R}_{>0}^{n}$
satisfying \textbf{IC1}, \textbf{IC2}, and \textbf{IC3}.

\item[Out:] $\Theta$, the set of equilibria satisfying \textbf{OC1}.
\end{description}

\begin{enumerate}
\item $\Theta\leftarrow\{\}$


\item For $\sigma\in{\left\{  -1,+1\right\}  ^{n}}$ do

\begin{enumerate}
\item $f_\sigma\leftarrow-R + \frac{1}{n}\sum_{\mu=1}^{n}\sigma_{\mu}\sqrt{k_{\mu
}^{2}R-\omega_{\mu}^{2}}$

\item $\mathcal{R}\leftarrow\mathbf{Solve}(f_\sigma,+)$

\item For $R\in\mathcal{R}$ do

\begin{enumerate}
\item Compute $\theta\in(-\pi,\pi]^{n}$ such that $\sin\theta_{\nu}=\frac{\omega_{\nu}}{k_{\nu
}\sqrt{R}}\hbox{~and~}\mathrm{sign}\cos\theta_{\nu}=\sigma_{\nu}$
for $\nu = 1,\dots,n$.

\item Add $\theta$ to $\Theta$
\end{enumerate}
\end{enumerate}
\end{enumerate}
\end{algorithm}

\begin{example}\label{ex:Illustrative1}
To illustrate for $n = 2$, consider $\omega=( 4, -4 )$ and $k=( 5, 2 )$.
There are $4$ sign patterns~$\sigma$ to consider:
\begin{itemize}
\item $\sigma = (-1,-1):$
\begin{itemize}
\item[$\circ$] $f_\sigma(R) = -R + \frac{1}{2}\left(-\sqrt{25 R - 16} - \sqrt{4 R - 16} \right)$ has no positive roots.
\end{itemize}
\item $\sigma=(-1,+1):$
\begin{itemize}
\item[$\circ$] $f_\sigma(R) = -R + \frac{1}{2}\left(-\sqrt{25 R - 16} + \sqrt{4 R - 16} \right)$ has no positive roots.
\end{itemize}
\item $\sigma=(+1,-1):$
\begin{itemize}
\item[$\circ$] $f_\sigma(R) = -R + \frac{1}{2}\left(\sqrt{25 R - 16} - \sqrt{4 R - 16} \right)$ has one positive root, namely $R = 4.25$.
\item[$\circ$] This yields the equilibrium $\theta=(0.3985, -1.8158)$.
\end{itemize}
\item $\sigma=(+1,+1):$
\begin{itemize}
\item[$\circ$] $f_\sigma(R) = -R + \frac{1}{2}\left(\sqrt{25 R - 16} + \sqrt{4 R - 16} \right)$ has one positive root, namely $R = 10.25$.
\item[$\circ$] This yields the equilibrium $\theta=(0.2526, -0.6747)$.
\end{itemize}
\end{itemize}
In summary, there are two equilibria satisfying \eqref{eq:Kuramoto_equilibrium}.
\end{example}

\subsection{Optimizations}\label{sec:Optimizations}

In Algorithm~\ref{alg:Basic}, $\mathbf{Solve}(f_\sigma,+)$,
which computed all positive roots of $f_\sigma$,
was called for all $2^n$ sign patterns. This exponential scaling in the number of oscillators is not much better than the previous approaches discussed earlier. As such, the goal of this section is to prune out sign patterns $\sigma$ for which $f_\sigma$ as in \eqref{eq:Univar}
has no positive roots. This improvement allows for 
the optimized algorithm to essentially scale based 
on the number of equilibria, provided an extra condition is satisfied, yielding much shorter computation times.

Throughout this section, we assume 
$\omega\in\mathbb{R}^{n}$ and $k\in\mathbb{R}_{>0}^{n}$
satisfy \textbf{IC1} -- \textbf{IC3},
$\sigma \in\{-1,+1\}^n$, and~$f_\sigma$ as in \eqref{eq:Univar}.
With this setup, the following provides an interval containing all positive roots~of~$f_\sigma$.

\begin{proposition}
\label{pro:opt1} 
If $\sigma_{+} = \{\mu : \sigma_{\mu}=+1\}$,
then every positive root of 
$f_{\sigma}$ is contained in the interval
\[
\left[  \left(  \frac{\omega_{n}}{k_{n}} \right)^{2},\;\; \left(  \frac{1}{n}
\sum_{\mu\in\sigma_{+}}k_{\mu} \right)^{2}~ \right].
\]

\end{proposition}

\begin{proof}
Suppose that $R$ is a positive root of $f_{\sigma}$. 
Since each
$\frac{\omega_{\nu}^{2}}{k_{\nu}^{2} R}
= \sin^{2}\theta_{\nu} \leq 1$ by Theorem~\ref{thm: reduction},
\[
R\geq\left(  \frac{\omega_{\nu}}{k_{\nu}} \right) ^{2} \hbox{~~~~for~}\nu=1,\ldots,n.
\]
Hence, \textbf{IC3} shows that $R\geq\left(  \dfrac{\omega_{n}}{k_{n}} \right) ^{2}$.

Moreover,
\[
0 \leq\sqrt{k_{\mu}^{2} R - \omega_{\mu}^{2}} \leq k_{\mu}\sqrt{R}.
\]
For $\sigma_{-} = \{ \mu:\sigma_{\mu
}=-1 \}$, we have
\begin{align*}
R  &  = \frac{1}{n}\sum_{\mu=1}^{n}\sigma_{\mu}\sqrt{k_{\mu}^{2} R -
\omega_{\mu}^{2}}\\
&  = \left(  \frac{1}{n} \sum_{\mu\in\sigma_{+}}\sqrt{k_{\mu}^{2} R -
\omega_{\mu}^{2}} \right)  -\left(  \frac{1}{n}\sum_{\mu\in\sigma_{-}}%
\sqrt{k_{\mu}^{2} R - \omega_{\mu}^{2}} \right) \\
&  \leq\frac{1}{n}\sum_{\mu\in\sigma_{+}}k_{\mu}\sqrt{R}.
\end{align*}
This is equivalent to
$R\leq\left(  \frac{1}{n}\sum_{\mu\in\sigma_{+}}k_{\mu} \right) ^{2}$.
\end{proof}

\begin{example}\label{ex:Illustrative2}
With the setup from Ex.~\ref{ex:Illustrative1}, we consider the four cases:
\begin{itemize}
\item $\sigma = (-1,-1):$
\begin{itemize}
\item[$\circ$] no positive roots since Prop.~\ref{pro:opt1} provides the ``interval'' $[4,0]$.
\end{itemize}
\item $\sigma=(-1,+1):$
\begin{itemize}
\item[$\circ$] no positive roots since Prop.~\ref{pro:opt1} provides
the ``interval'' $[4,1]$.
\end{itemize}
\item $\sigma=(+1,-1):$
\begin{itemize}
\item[$\circ$] Prop.~\ref{pro:opt1} provides the interval $[4,6.25]$, which contains the positive root $R = 4.25$.
\end{itemize}
\item $\sigma=(+1,+1):$
\begin{itemize}
\item[$\circ$] Prop.~\ref{pro:opt1} provides the interval $[4,12.25]$, which contains the positive root $R = 10.25$.
\end{itemize}
\end{itemize}
\end{example}

As shown in Ex.~\ref{ex:Illustrative2}, Prop.~\ref{pro:opt1}
can exclude sign patterns $\sigma$ for which 
$f_\sigma$ has no positive roots.  The following provides
another such test.

\begin{proposition}
\label{pro:opt2} If, for all $\ell=1,2,\ldots,n$,
\begin{equation}\label{eq:partialsum}
s_{\ell} = \sum_{\mu=1}^{\ell}\sigma_{\mu}k_{\mu}\leq0,
\end{equation}
then $f_{\sigma}$ has no positive roots.
\end{proposition}

\begin{proof}
Suppose that $R$ is a positive root of $f_\sigma$.
Then, by Prop.~\ref{pro:opt1} and \textbf{IC3},
$$R \geq\left(  \dfrac{\omega_{n}}{k_{n}} \right) ^{2}
\geq \left(  \dfrac{\omega_{n-1}}{k_{n-1}} \right) ^{2}
\geq \cdots \geq \left(  \dfrac{\omega_{1}}{k_{1}} \right) ^{2}$$ so that 
\[
\sqrt{R - \left(  \frac{\omega_{1}}{k_{1}} \right) ^{2} } \geq\cdots\geq
\sqrt{R - \left(  \frac{\omega_{n}}{k_{n}} \right) ^{2} } \geq 0.
\]
Thus, for every $\ell = 2,\dots,n$, we have
$s_\ell = s_{\ell-1} + \sigma_{\ell}k_\ell\leq0$ by definition and
$$
 0 \geq s_\ell\cdot \sqrt{R-\left(\frac{\omega_\ell}{k_\ell}\right)^2} \geq 
s_\ell\cdot \sqrt{R-\left(\frac{\omega_{\ell-1}}{k_{\ell-1}}\right)^2}.
$$
Hence, combining with $f_\sigma(R) = 0$, we have
\begin{align*}
0 &\geq  s_n\cdot \sqrt{R - \left(  \frac
{\omega_{n}}{k_{n}} \right) ^{2} } \\
&= (s_{n-1} + \sigma_n k_n)\cdot \sqrt{R - \left(  \frac
{\omega_{n}}{k_{n}} \right) ^{2} } \\
&= s_{n-1}\cdot \sqrt{R - \left(  \frac
{\omega_{n}}{k_{n}} \right) ^{2} } + \sigma_{n} k_{n}\cdot \sqrt{R - \left(
\frac{\omega_{n}}{k_{n}} \right) ^{2}}\\
&\geq s_{n-1}\cdot \sqrt{R - \left(  \frac
{\omega_{n-1}}{k_{n-1}} \right) ^{2} } + \sigma_{n} k_{n}\cdot \sqrt{R - \left(
\frac{\omega_{n}}{k_{n}} \right) ^{2}}\\
&= (s_{n-2} + \sigma_{n-1}k_{n-1})\cdot \sqrt{R - \left(  \frac
{\omega_{n-1}}{k_{n-1}} \right) ^{2} } 
+ \sigma_{n} k_{n}\cdot \sqrt{R - \left(
\frac{\omega_{n}}{k_{n}} \right) ^{2}}\\
& \geq s_{n-2} \cdot \sqrt{R - \left(  \frac
{\omega_{n-2}}{k_{n-2}} \right) ^{2} }  + \sum_{\mu = n-1}^n \sigma_\mu k_\mu \cdot \sqrt{R-\left(\frac{\omega_\mu}{k_\mu}\right)^2}\\
& \vdots  \\
&\geq  \sum_{\mu=1}^{n}\sigma_{\mu}k_{\mu}\sqrt{R - \left( \frac{\omega_{\mu
}}{k_{\mu}} \right) ^{2}} \\
& = nR.
\end{align*}
This is a contradiction since $R > 0$.
\end{proof}

\begin{example}\label{ex:Illustrative3}
With the setup from Ex.~\ref{ex:Illustrative1}, Prop.~\ref{pro:opt2} shows that $f_\sigma$ can have no positive
roots for $\sigma = (-1,-1)$ and $\sigma=(-1,+1)$.
\end{example}

The following will be used to show additional conditions 
for which $f_\sigma$ has no positive roots.

\begin{lemma}
\label{lem:opt1} $f_{\sigma}$ has no positive roots if and only if $f_{\sigma} < 0$
on $\left[  \left( \dfrac{\omega_{n}}{k_{n}}\right) ^{2}, \infty\right) $.
\end{lemma}

\begin{proof}
Let $I = \left[  \left( \dfrac{\omega_{n}}{k_{n}}\right) ^{2}, \infty\right) $. If $R \in I$, then
\begin{align*}
f_{\sigma}(R)  &  = -R + \frac{1}{n}\sum_{\mu=1}^{n}\sigma_{\mu} \sqrt{
k_{\mu}^{2} R - \omega_{\mu}^{2} }\\
&  \leq-R + \frac{1}{n}\sum_{\mu=1}^{n} \sqrt{ k_{\mu}^{2} R - \omega_{\mu
}^{2} }\\
&  \leq-R + \frac{1}{n}\sum_{\mu=1}^{n} k_{\mu}\sqrt{R}
\end{align*}
so that
\[
\lim\limits_{R\rightarrow\infty} f_{\sigma}(R) = -\infty.
\]
Since $f_{\sigma}$ is continuous on $I$,
we must have $f_{\sigma} < 0$ on $I$ when $f_\sigma$ has no positive roots. Furthermore, the interval $\left[  \left(  \frac{\omega_{n}}{k_{n}} \right)^{2},\;\; \left(  \frac{1}{n}
\sum_{\mu\in\sigma_{+}}k_{\mu} \right)^{2}~ \right]$ contains all the positive roots of $f_{\sigma}$ by Prop.~\ref{pro:opt1} and is contained in $I$. Hence, if $f_{\sigma} < 0$ on $I$, then $f_{\sigma}$ has no positive roots.
\end{proof}

If $f_\sigma$ has no positive roots, the following
yields additional cases which also have no~positive~roots.

\begin{lemma}
\label{pro:opt3} Let $\mu$ be such that $\sigma_{\mu}=+1$. 
Let $\sigma^{\prime}$ such that 
$\sigma_{\mu}^{\prime}=-1$ and 
$\sigma^\prime_\nu = \sigma_\nu$ for $\nu \neq \mu$.
If~$f_{\sigma}$ has no positive roots, 
then $f_{\sigma^{\prime}}$ also has no positive roots.
\end{lemma}

\begin{proof}
Since $f_{\sigma}$ has no positive roots, 
Lemma~\ref{lem:opt1} shows that $f_\sigma < 0$
on $\left[  \left(
\dfrac{\omega_{n}}{k_{n}} \right) ^{2}, \infty\right) $.
Since $f_{\sigma^{\prime}} \leq f_{\sigma}$
on $\left[  \left(
\dfrac{\omega_{n}}{k_{n}} \right) ^{2}, \infty\right) $,
$f_{\sigma^{\prime}}$ also does not have any positive roots by Lemma~\ref{lem:opt1}.
\end{proof}

\begin{example}\label{ex:Illustrative4}
With the setup from Ex.~\ref{ex:Illustrative1}, 
since $f_\sigma$ for $\sigma = (-1,+1)$ has no positive roots,
$f_{\sigma'}$ also has no positive roots for $\sigma' = (-1,-1)$.
\end{example}

The following excludes additional cases by swapping entries of $\sigma$.

\begin{lemma}
\label{pro:opt4} Suppose $\mu$ and $\nu$ are such that $\sigma_{\mu}=+1$ and
$\sigma_{\nu}=-1$. Let $\sigma^{\prime}$ be the same as $\sigma$ except that
$\sigma_{\mu}^{\prime}=\sigma_\nu = -1$ and $\sigma_{\nu}^{\prime}=\sigma_\mu = +1$. If $f_{\sigma}$ has
no positive roots where
\begin{equation}\label{eq:Prop4}
\left(  k_{\mu}^{2}-k_{\nu}^{2} \right) \left(  \frac{\omega_{n}}{k_{n}}
\right) ^{2} \geq\omega_{\mu}^{2} - \omega_{\nu}^{2} \hbox{~~~~and~~~~} \left(
k_{\mu}^{2}-k_{\nu}^{2} \right) \left(  \frac{1}{n}\sum_{\iota=1}^{n}k_{\iota}
\right) ^{2} \geq\omega_{\mu}^{2} - \omega_{\nu}^{2},
\end{equation}
then $f_{\sigma^{\prime}}$ also has no positive roots.
\end{lemma}

\begin{proof}
Given \eqref{eq:Prop4} and $R\in\left[  \left(  \dfrac{\omega_{n}}{k_{n}} \right)
^{2}, \left(  \displaystyle\frac{1}{n}\sum_{\mu=1}^{n}k_{\mu} \right) ^{2}~ \right]$, we have
\[
\left(  k_{\mu}^{2}-k_{\nu}^{2} \right) R \geq\omega_{\mu}^{2} - \omega_{\nu
}^{2}.
\]
Rearranging gives
\[
k_{\mu}^{2}R - \omega_{\mu}^{2} \geq k_{\nu}^{2}R - \omega_{\nu}^{2} \geq 0
\hbox{~~so that~~}
\sqrt{k_{\mu}^{2}R - \omega_{\mu}^{2}} \geq\sqrt{k_{\nu}^{2}R - \omega_{\nu
}^{2}}.
\]
Hence, $f_{\sigma^{\prime}}(R) \leq f_{\sigma}(R)$.
Therefore, the result follows from Lemma~\ref{lem:opt1}.
\end{proof}

A natural way to order the sign patterns $\sigma$ is
to present them using binary representations of the 
numbers in base $10$ from $0$ to $2^{n}-1$ 
where ``0'' in binary represents $-1$ and ``1'' 
in binary represents~$+1$.  For example, $\sigma = (+1,-1)$
corresponds to the binary number $10_2$,
so we can say~$\sigma$ corresponds to the number $2$ in base $10$.
We demonstrate this on a concrete application (power flow analysis) 
from electrical engineering~\cite{dorfler2013} and apply all the previous results.

\begin{example}[Power flow model: 4-bus system]\label{ex:PowerFlow}
Figure~\ref{fig:fourbussystem} depicts a lossless four-bus power system with active power injections $P_{1},\ldots,P_{4}$, voltage magnitudes $\left| V_{1}\right|,\ldots,\left| V_{4}\right| $, and line
susceptances $b_{12} = b_{13} = b_{14} = b_{23} = b_{24} = b_{34} = -1$. The equilibria of the power flow equations correspond to the equilibria of the rank-one coupled Kuramoto model, namely  the  solutions 
of~\eqref{eq:Kuramoto_equilibrium}
 where   $\omega= \left( P_{1},\ldots, 
P_{4}\right) $,  $k = (2|V_{1}|, 2|V_{2}|, 2|V_{3}|, 2|V_{4}|)$,  and $\theta= \left( \theta
_{1}, \theta_{2}, \theta_{3}, \theta_{4}\right)$
are the voltage angles. 

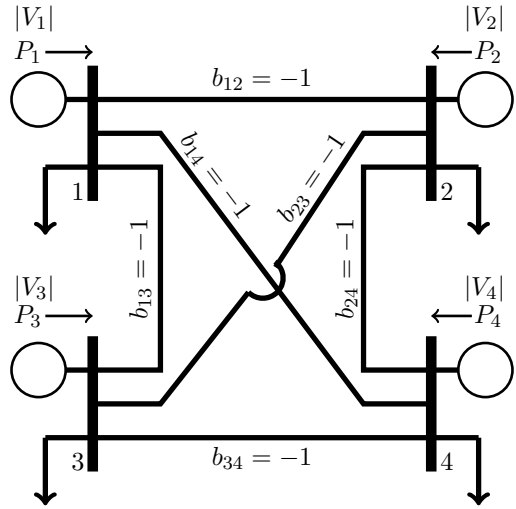
\begin{figure}[th]
\centering
\begin{circuitikz}[scale=0.9, transform shape]
\ctikzset{bipoles/length=0.6cm}
\path[draw,line width=4pt] (0,0) -- (0,2);
\draw (-0.45,0.4) node[below right] {$1$};
\draw (-1.3,2.7) node[right] {$\left|V_1\right|$};
\draw (-1.3,2.2) node[right] {$P_1$};
\path[draw,->,line width=1pt] (-0.7,2.2) -- (0,2.2);
\path[draw,line width=2pt] (-0.4,1.5) -- (0,1.5);
\path[draw,line width=2pt] (-0.7,0.5) -- (0,0.5);
\path[draw,->,line width=2pt] (-0.7,0.5) -- (-0.7,-0.5);
\draw[line width=1] (-0.8,1.5) circle (0.4);
\path[draw,line width=4pt] (5,0) -- (5,2);
\draw (5,0.4) node[below right] {$2$};
\path[draw,line width=2pt] (5,1.5) -- (5.4,1.5);
\path[draw,line width=2pt] (5,0.5) -- (5.7,0.5);
\path[draw,->,line width=2pt] (5.7,0.5) -- (5.7,-0.5);
\draw[line width=1] (5.8,1.5) circle (0.4);
\draw (5.35,2.7) node[right] {$\left| V_2 \right|$};
\draw (5.5,2.2) node[right] {$P_2$};
\path[draw,->,line width=1pt] (5.6,2.2) -- (5,2.2);
\path[draw,line width=4pt] (0,-4) -- (0,-2);
\draw (-0.45,-3.6) node[below right] {$3$};
\draw (-1.3,-1.3) node[right] {$\left|V_3\right|$};
\draw (-1.3,-1.7) node[right] {$P_3$};
\path[draw,->,line width=1pt] (-0.7,-1.7) -- (0,-1.7);
\path[draw,line width=2pt] (-0.4,-2.5) -- (0,-2.5);
\path[draw,line width=2pt] (-0.7,-3.5) -- (0,-3.5);
\path[draw,->,line width=2pt] (-0.7,-3.5) -- (-0.7,-4.5);
\draw[line width=1] (-0.8,-2.5) circle (0.4);
\path[draw,line width=4pt] (5,-4) -- (5,-2);
\draw (5,-3.6) node[below right] {$4$};
\path[draw,line width=2pt] (5,-2.5) -- (5.4,-2.5);
\draw[line width=1] (5.8,-2.5) circle (0.4);
\draw (5.35,-1.3) node[right] {$\left|V_4\right|$};
\draw (5.5,-1.7) node[right] {$P_4$};
\path[draw,->,line width=1pt] (5.6,-1.7) -- (5,-1.7);
\path[draw,line width=2pt] (5,-3.5) -- (5.7,-3.5);
\path[draw,->,line width=2pt] (5.7,-3.5) -- (5.7,-4.5);
\path[draw,line width=2pt] (0,1.5) -- (5,1.5);
\draw (2.5,1.5) node[above] {$b_{12} = -1$};
\path[draw,line width=2pt] (0,-3.5) -- (5,-3.5);
\draw (2.5,-3.5) node[below] {$b_{34} = -1$};
\path[draw,line width=2pt] (0,0.5) -- (1,0.5) -- (1,-2.5) -- (0,-2.5);
\draw (1,-1) node[rotate=90,anchor=south] {$b_{13} = -1$};
\path[draw,line width=2pt] (5,0.5) -- (4,0.5) -- (4,-2.5) -- (5,-2.5);
\draw (3.45,-1) node[rotate=90,anchor=north] {$b_{24} = -1$};
\path[draw,line width=2pt] (0,1) -- (1,1) -- (4,-3) -- (5,-3);
\draw (2.05,0.55) node[rotate=-53,anchor=north] {$b_{14} = -1$};
\path[draw,line width=2pt] (0,-3) -- (1,-3) -- (2.3,-1.35);
\draw[line width=2pt] (2.3,-1.35) arc (-135:53:0.3);
\path[draw,line width=2pt] (2.7,-0.98) -- (4,1) -- (5,1);
\draw (3,0.5) node[rotate=--56,anchor=north] {$b_{23} = -1$};
\end{circuitikz}
\caption{One-Line Diagram for a Four-Bus Electric Power System}
\label{fig:fourbussystem}
\end{figure}

Let us consider the case 
with
$P=(1.00,-1.25,2.00,-1.75)$ and $|V|=(1.10,0.93,1.05,0.90)$. 
That is, we aim to 
solve~\eqref{eq:Kuramoto_equilibrium} where 
$\omega=(1.00,-1.25,2.00,-1.75)$ and $k=(2.20, 1.86, 2.10, 1.80)$. 
By taking
the $16$ possible sign patterns as the numbers \[0\equiv(-1,-1,-1,-1),~\dots
,~15\equiv(+1,+1,+1,+1),\] some possibilities can immediately be ruled out:
\begin{itemize}
\item $0$, $1$, $2$, $4$, $5$, $8$ by Prop.~\ref{pro:opt1};
\item $0$, $1$, $2$, $3$, $4$, $5$ by Prop.~\ref{pro:opt2}.
\end{itemize}
We now consider the remaining possibilities starting from the largest:
\begin{itemize}
\item One equilibria resulting from each of the following: $15$, $14$, $13$,
$12$, $11$, $10$;
\item No equilibria resulting from $9\equiv(+1,-1,-1,+1)$;
\item Two equilibria resulting from $7\equiv(-1,+1,+1,+1)$;
\item No equilibria resulting from $6\equiv(-1,+1,+1,-1)$.
\end{itemize}
For example, since $9\equiv(+1,-1,-1,+1)$ yields no equilibria,
Lemma~\ref{pro:opt3} provides that $8$, $1$, and~$0$ 
also yield no equilibria while
Lemma~\ref{pro:opt4} provides that $5$ yields no equilibria.
In summary, 
this particular case has 
a total of eight equilibria satisfying \eqref{eq:Kuramoto_equilibrium}.
\end{example}

We now turn to consider a special case for which we 
can provide further optimizations:
\begin{enumerate}
\item[\textbf{IC4}:] $k_{1} \geq k_{2} \geq\cdots\geq k_{n}$
\end{enumerate}
We note that \textbf{IC4} is  independent  of the implicitly assumed conditions \textbf{IC1}--\textbf{IC3}
so we will explicitly state when this condition is also required.
With \textbf{IC4}, we provide 
a simplification of Lemma~\ref{pro:opt4}.

\begin{lemma}\label{pro:opt5}
Suppose that \textbf{IC4} is satisfied and $\mu$ and $\nu$ are 
such that $\sigma_{\mu}=+1$ and $\sigma_{\nu}=-1$. 
Let $\sigma^{\prime}$ be the same as $\sigma$ except that 
$\sigma_{\mu}^{\prime}=\sigma_\nu=-1\ $ and $\sigma_{\nu}^{\prime
}=\sigma_\mu=+1$. If $\mu< \nu$ and $f_{\sigma}$ has no positive roots, 
then $f_{\sigma^{\prime}}$ also has no positive roots.
\end{lemma}

\begin{proof}
From \textbf{IC4}, and $\mu < \nu$, we have
\[
k_{\mu} \geq k_{\nu} \;\text{ and } \left(  \frac{\omega_{\mu}}{k_{\mu}}
\right) ^{2} \leq\left(  \frac{\omega_{\nu}}{k_{\nu}} \right) ^{2}.
\]
For $R\geq \left(\dfrac{\omega_n}{k_n}\right)^2$,
\[
k_{\mu} \sqrt{R - \left(  \frac{\omega_{\mu}}{k_{\mu}} \right) ^{2}} \geq
k_{\nu} \sqrt{R - \left(  \frac{\omega_{\nu}}{k_{\nu}} \right) ^{2}}
\hbox{~~so that~~}
\sqrt{k_{\mu}^{2} R - \omega_{\mu}^{2}} \geq\sqrt{k_{\nu}^{2} R - \omega_{\nu
}^{2}}.
\]
Hence, $f_{\sigma^{\prime}}(R) \leq f_{\sigma}(R)$.
Therefore, the result follows from Lemma~\ref{lem:opt1}.
\end{proof}

Writing $\sigma$ as a binary number, Lemma~\ref{pro:opt3} 
allows changing a ``1'' to a ``0.'' 
With \textbf{IC4}, Lemma~\ref{pro:opt5} 
allows swapping a ``0'' and a ``1'' provided the ``0''
is on the right of ``1.''  Thus, with this understanding and ordering, we state the main optimization result.

\begin{theorem}
\label{cor:opt1} Assume that {\rm\textbf{IC4}} is satisfied.
\begin{enumerate}
\item\label{case:1} Suppose $\sigma = (+1,\dots,+1)$ and $f_{\sigma}$ has no positive roots.
Then, for every $\sigma'\in\{-1,+1\}^n$, $f_{\sigma'}$ has no positive roots.
\item\label{case:2} Suppose $\sigma$ has exactly one entry which is $-1$ and
that $f_{\sigma}$ has no positive roots.
Then, $f_{\sigma'}$ also has no positive roots for every
$\sigma'\in\{-1,+1\}^n$ which is smaller than $\sigma$ using the aforementioned binary representation.
\item Suppose that $\sigma$ has at least two entries equal to $-1$
and $f_\sigma$ has no positive roots.  Let $\ell$ be the penultimate 
entry of a~$-1$ in $\sigma$.
Let $\sigma = (\rho_1,\rho_2)$ where $\rho_1 = (\sigma_1,\dots,\sigma_{\ell-1},-1)$
and $\rho_2 = (\sigma_{\ell+1},\dots,\sigma_n)$.
Then, for every \mbox{$\sigma' = (\rho_1,\rho_2^\prime)\in\{-1,+1\}^n$}
such that $\rho_2^\prime$ is smaller than~$\rho_2$ using the aforementioned binary representation, $f_{\sigma'}$ also has no positive roots.
\end{enumerate}
\end{theorem}

\begin{proof} We prove the three cases as follows.
\begin{enumerate}
        \item This case follows immediately by repeated application of Lemma~\ref{pro:opt3}.
        
   \item This case follows by alternately applying
   Case~\ref{case:1} to parts of $\sigma$ and Lemma~\ref{pro:opt5}.
        
        \item This case follows by applying Case~\ref{case:2} to $\rho_2$.
\end{enumerate}
\end{proof}

The main benefit of this theorem is that it allows one to skip {\em sequential} sign cases by directly computing the next case that needs to be checked from the current case.

\begin{example}
To illustrate, suppose the input parameters satisfy {\rm\textbf{IC4}} and $f_\sigma$ has no positive roots 
for $\sigma= \left( +1, +1, -1, +1, -1, +1 \right) \equiv 110101_2 = 53$.
Theorem~\ref{cor:opt1} shows that $f_{\sigma'}$ also
has no positive roots for the following sequential sign
patterns $\sigma'$:
\begin{align*}
&  \left( +1, +1, -1, +1, -1, -1 \right) \equiv 110100_2 = 52 \\
&  \left( +1, +1, -1, -1, +1, +1 \right) \equiv 110011_2 = 51 \\
&  \left( +1, +1, -1, -1, +1, -1 \right) \equiv 110010_2 = 50\\
&  \left( +1, +1, -1, -1, -1, +1 \right) \equiv 110001_2 = 49\\
&  \left( +1, +1, -1, -1, -1, -1 \right) \equiv 110000_2 = 48.
\end{align*}
Furthermore, 48 can be immediately calculated from 53 by zeroing out everything from the next to last 0 onward, so that the five listed cases do not need to be considered at all.
\end{example}

The following utilizes these previous results 
assuming \textbf{IC4} to more efficiently
compute the set of all equilibria to \eqref{eq:Kuramoto_equilibrium}.
This depends on two algorithms:
a root finding method that returns the set of all roots of $f_\sigma$
in an interval $I$, denoted $\mathbf{Solve}(f_\sigma,I)$,
and a method that returns a sign pattern in $\{-1,+1\}^n$
given a number $0\leq\iota\leq 2^n-1$, denoted
$\mathbf{Convert}(\iota)$.

\begin{algorithm}
[Optimized]\label{alg:Optimized}\ 

\begin{description}
[leftmargin=3em,style=nextline,itemsep=0.5em]

\item[In:] $\omega\in\mathbb{R}^{n}$ and $k\in\mathbb{R}_{>0}^{n}$ satisfying
{\rm \textbf{IC1}--\textbf{IC4}}.

\item[Out:] $\Theta$, the set of equilibria satisfying {\rm\textbf{OC1}}.
\end{description}

\begin{enumerate}
\item $\Theta\leftarrow\{\}$

\item $\iota\leftarrow2^{n}-1$

\item\label{item:while} While $\iota\geq0$ do

\begin{enumerate}

\item $\sigma\leftarrow \mathbf{Convert}(\iota)$

\item\label{item:3b} $I \leftarrow\left[  \left(  \dfrac{\omega_{n}}{k_{n}} \right) ^{2}, \;\;
\left( \displaystyle \frac{1}{n}\sum_{\mu\in\sigma_{+}}k_{\mu} \right) ^{2} ~ \right] $

\item\label{item:3c} If $I$ is empty, then

\begin{enumerate}

\item Decrement $\iota$ according to Theorem~\ref{cor:opt1}

\item Continue (go back to the start of Step~\ref{item:while})
\end{enumerate}

\item\label{item:3d} If $\sum_{\mu=1}^{\ell} \sigma_{\mu
}k_{\mu} \leq0$ for all $\ell= 1,2, \ldots, n$,  then

\begin{enumerate}

\item Decrement $\iota$ according to Theorem~\ref{cor:opt1}.

\item Continue (go back to the start of Step~\ref{item:while})
\end{enumerate}

\item $f_\sigma \leftarrow-R + \frac{1}{n} \sum_{\mu=1}^{n} \sigma_{\mu}\sqrt{k_{\mu
}^{2} R - \omega_{\mu}^{2}}$

\item $\mathcal{R}\leftarrow\mathbf{Solve}(f_\sigma,I)$

\item\label{item:3g} If $\mathcal{R} = \emptyset$, then

\begin{enumerate}

\item Decrement $\iota$ according to Theorem~\ref{cor:opt1}

\item Continue (go back to the start of Step~\ref{item:while})
\end{enumerate}

\item For $R\in\mathcal{R}$ do

\begin{enumerate}
\item Compute $\theta\in(-\pi,\pi]^{n}$ such that $\sin\theta_{\nu}=\frac{\omega_{\nu}}{k_{\nu
}\sqrt{R}}\hbox{~and~}\mathrm{sign}\cos\theta_{\nu}=\sigma_{\nu}$
for $\nu = 1,\dots,n$.

\item Add $\theta$ to $\Theta$
\end{enumerate}

\item $\iota\leftarrow\iota-1$
\end{enumerate}
\end{enumerate}
\end{algorithm}

\begin{remark}
In Algorithm~\ref{alg:Optimized},
Steps~\ref{item:3b} and~\ref{item:3d} follow from Prop.~\ref{pro:opt1} 
and~\ref{pro:opt2}, respectively.
\end{remark}

\begin{example}\label{ex:Illustrative5}
To illustrate, we apply Algorithm~\ref{alg:Optimized}
to the setup from Ex.~\ref{ex:Illustrative1}.
\begin{itemize}
\item $\iota = 3$ yielding $\sigma = (+1,+1)$:
\begin{itemize}
\item[$\circ$] $I = [4, 12.25]$ 
\item[$\circ$] One positive root of $f_\sigma(R) = -R + \frac{1}{2}\left(\sqrt{25 R - 16} + \sqrt{4 R - 16} \right)$ on $I$, namely $R = 10.25$.
\item[$\circ$] This yields the equilibrium $\theta=(0.2526, -0.6747)$.
\end{itemize}
\item $\iota = 2$ yielding $\sigma=(+1,-1)$:
\begin{itemize}
\item[$\circ$] $I = [4, 6.25]$ 
\item[$\circ$] One positive root of $f_\sigma(R) = -R + \frac{1}{2}\left(\sqrt{25 R - 16} - \sqrt{4 R - 16} \right)$ in $I$, namely $R = 4.25$.
\item[$\circ$] This yields the equilibrium $\theta=(0.3985, -1.8158)$.
\end{itemize}
\item $\iota = 1$ yielding $\sigma=(-1,+1)$:
\begin{itemize}
\item[$\circ$] $I = [4, 1]$ is empty
\item[$\circ$] Theorem~\ref{cor:opt1} removes the $\iota = 0$ case.
\end{itemize}
\end{itemize}
In summary, there are two equilibria satisfying \eqref{eq:Kuramoto_equilibrium}.
\end{example}

\subsection{Performance}\label{sec:Performance}

We implemented both Algorithms~\ref{alg:Basic}
and~\ref{alg:Optimized} in C++ using the C-XSC library \cite{CXSC} with the univariate solver
being an interval Newton method \cite[Chap.~6]{hammer1995c++}. The implementation is available at \url{http://dx.doi.org/10.7274/R09W0CDP}.
In this section, we benchmark the performance of this 
with the following methods for computing all equilibria to
\eqref{eq:Kuramoto_equilibrium}:
\begin{itemize}
\item solve \eqref{eq:Kuramoto_poly} using Gr\"obner basis
techniques in {\tt Macaulay2} \cite{M2};
\item solve \eqref{eq:Kuramoto_poly} using 
homotopy continuation in {\tt Bertini} \cite{BHSW06}
as in \cite{mehta2015algebraic};
\item compute equilibria for \eqref{eq:Kuramoto_equilibrium}
using elliptical continuation from~\cite{lesieutre_wu_allerton2015}.
\end{itemize}
We end with an example having 
$n=60$ that is easily 
solvable using Algorithm~\ref{alg:Optimized}.

\paragraph{Comparison with computational algebraic geometry:}

We use the following 
setup from~\cite{mehta2015algebraic} to compare with 
solving \eqref{eq:Kuramoto_poly}
using {\tt Macaulay2} and {\tt Bertini}
with serial computations.  For each $n = 3,\dots,12$,
the natural frequencies are equidistant, namely
$\omega_\mu = -1 + (2\mu-1)/n$ for $\mu = 1,\dots,n$,
with uniform coupling $k = (\sqrt{1.5},\dots,\sqrt{1.5})$.
To simplify the algebraic geometry computations
using {\tt Macaulay2} and {\tt Bertini}, we compute
the equilibria as in \cite{mehta2015algebraic}
by setting $\theta_n = 0$ ($s_n = 0$ and $c_n = 1$)
with the results summarized in Table~\ref{tab:comparison}.

With {\tt Macaulay2}, we simply computed
the total number of complex solutions,
i.e.,~the degree
of the ideal generated by the polynomials in~\eqref{eq:Kuramoto_poly}
when $s_n = 0$ and $c_n = 1$.  
Thus, one would need
to perform additional computations to compute
the number of real solutions.  
The symbol $\ddagger$ means that the computation did not
complete within 48 hours.

With {\tt Bertini}, we performed two
different computations.  The first was 
to directly solve \eqref{eq:Kuramoto_poly}
using regeneration~\cite{Regeneration}
and the second utilized a parameter
homotopy~\cite{Parameter}.
Both of these computations 
provide all real and non-real solutions to \eqref{eq:Kuramoto_poly}.

Although {\tt Bertini} is parallelized
and Algorithm~\ref{alg:Optimized} is parallelizable,
we again note that the data in Table~\ref{tab:comparison}
is based on using serial processing.
Nonetheless, this shows the advantage
of using Algorithm~\ref{alg:Optimized}
to compute all equilibria without needing 
to compute the non-real 
solutions~of~\eqref{eq:Kuramoto_poly}.

\begin{table}[!ht]
{\scriptsize
\begin{center}
\begin{tabular}{c|c|c|c|c|c|c|c|c|c|c} 
$n$ & $3$ & $4$ & $5$ & $6$ & $7$ & $8$ & $9$ & $10$ & $11$ &  $12$ \\
\hline
\hline
\# real & $2$ & $2$ & $4$ & $4$ & $4$ & $4$ & $4$ & $4$ & $8$ & $8$\\
\hline
\# complex & $6$ & $12$ & $28$ & $56$ & $118$ & $238$ & $486$ & $976$ & $1972$ & $3958$ \\
\hline
\hline
{\tt Macaulay2} degree & $<0.1$s & $<0.1$s & $0.1$s & $1.1$s & $7.0$s & $72.6$s & $716.5$s & $10783.7$s & $149578.0$s & $\ddagger$ \\
\hline
{\tt Bertini} regeneration & $0.3$s & $1.2$s & $3.4$s & $13.4$s & $45.1$s & $116.6$s & $210.1$s & $486.2$s & $1493.1$s & $3443.5$s\\
\hline
{\tt Bertini} parameter & $<0.1$s & $<0.1$s & $0.2$s & $0.4$s & $1.1$s & $2.2$s & $6.9$s & $15.0$s & $36.9$s & $116.8$s\\
\hline
Algorithm~\ref{alg:Optimized} & $<0.1$s & $<0.1$s & $<0.1$s & $<0.1$s & $<0.1$s & $<0.1$s & $<0.1$s & $<0.1$s & $<0.1$s & $<0.1$s
\end{tabular}
\end{center}
\caption{Comparison of various solving methods
}\label{tab:comparison}
}
\end{table}

\paragraph{Comparison with elliptical continuation:}

We next compare Algorithm~\ref{alg:Optimized} with the elliptical continuation method proposed in~\cite{lesieutre_wu_allerton2015}. While having the advantage of being applicable to a 
more general setting of power flow equations, 
the elliptical continuation method in~\cite{lesieutre_wu_allerton2015} comes with both theoretical and computational drawbacks relative to Algorithm~\ref{alg:Optimized} when considered in the context of the Kuramoto model. 
In contrast to Algorithm~\ref{alg:Optimized}, 
there currently is no theoretical guarantee that
the elliptical continuation method in~\cite{lesieutre_wu_allerton2015} will 
compute all equilibria.
Moreover, the computational speed of Algorithm~\ref{alg:Optimized} can be several orders of magnitude faster than the elliptical continuation method in~\cite{lesieutre_wu_allerton2015}. Consider, for instance, a test case with $n = 18$, $k = (1,\dots,1)$, and 
$$
\begin{array}{rcl}
\omega &=&
\left( 0.1000,\, -0.1000,\, -0.1415,\, -0.1429,\, 0.1500,\, 0.2000,\, -0.4142,\, 0.7000,\, -0.8500,\, \right. \\
& &  ~\left. 1.4142,\, 2.3000,\, 3.1415,\, -3.1904,\, -3.5000,\, 4.3333,\, -5.0000,\, -6.0000,\, 7.0000 \right).
\end{array}
$$
When interpreted as a power flow problem, this test case represents a power system composed of~$18$ buses with fixed, unity voltage magnitudes and specified active power injections given by $\omega$ in normalized ``per unit'' values. The buses are completely connected by lines with unity reactance and zero resistance. While this is a very special example of a power system network, the corresponding test case enables comparison between Algorithm~\ref{alg:Optimized} and the elliptical continuation method in~\cite{lesieutre_wu_allerton2015} in the context of the Kuramoto model.

A serial implementation of the elliptical continuation method in~\cite{lesieutre_wu_allerton2015} in M{\sc atlab}
yielded $8538$ equilibria satisfying \eqref{eq:Kuramoto_equilibrium}
in $1.935\times10^5$~seconds ($53.77$~hours).
For a fair comparison, we used a serial
implementation of Algorithm~\ref{alg:Optimized} 
in M{\sc atlab} which computed $8538$ 
equilibria in~$13.9$ seconds.
Hence, the implementation of Algorithm~\ref{alg:Optimized} in M{\sc atlab}
is roughly four orders of magnitude faster
than the M{\sc atlab} implementation of 
\cite{lesieutre_wu_allerton2015} for this example.
We note that the~C++ implementation of Algorithm~\ref{alg:Optimized} took $6.6$ seconds.

\paragraph{An example with $n = 60$:}

We conclude with an example solved
by Algorithm~\ref{alg:Optimized} for $n = 60$
having $k = (60,\dots,60)$ and 
\begin{center}
       \begin{tabular}{ c c c c c c c c c c c c}
                $\omega = $( & $0$, & $0$, & $0$, & $0$, & $0$, & $0$, & $0$, & $0$, & $0$, & $20$,& \\
                & $-20$, & $40$, & $-60$, & $60$, & $60$, & $80$, & $-80$, & $-100$, & $-100$, & $120$,& \\
                & $-160$, & $-160$, & $-200$, & $240$, & $-280$, & $-300$, & $300$, & $-360$, & $360$, & $-380$,& \\
                & $420$, & $420$, & $-420$, & $-460$, & $460$, & $500$, & $520$, & $540$, & $-560$, & $-600$,& \\
                & $-620$, & $620$, & $-640$, & $660$, & $660$, & $660$, & $680$, & $-720$, & $780$, & $-800$,& \\
                & $820$, & $-820$, & $-840$, & $-840$, & $-880$, & $920$, & $-980$, & $-980$, & $-1080$, & $3500$& $)$.
       \end{tabular}
 \end{center}
This example has 2 equilibria 
satisfying~\eqref{eq:Kuramoto_equilibrium} 
with the total computation time 
using the C++ implementation of Algorithm~\ref{alg:Optimized}
taking under a second. For comparison, the elliptical continuation method as described in the previous example took $5609$ seconds ($93.5$ minutes). This example is simply too large for 
current methods that compute all complex roots.
Generally, problems with~$\left( \dfrac{w_n}{k_n} \right)^2$ near $\left( \dfrac{1}{n} \displaystyle\sum_{\mu=1}^{n}k_{\mu} \right)^2$ will be solved quickly by Algorithm~\ref{alg:Optimized} 
as a consequence of Prop.~\ref{pro:opt1} and Theorem~\ref{cor:opt1}.

\section{Counting equilibria}\label{sec:UpperBound}

After reviewing known information, we compute
the generic root count for \eqref{eq:Kuramoto_poly}
which bounds the number of equilibria to \eqref{eq:Kuramoto_equilibrium}.
By analyzing the number of equilibria
in particular cases, we can
asymptotically compare
the maximum number of equilibria
to the generic root count of \eqref{eq:Kuramoto_poly}.

\subsection{Summary of known results}\label{sec:SmallCases}

As mentioned in the
Introduction, the arbitrary coupling
case \eqref{eq:Kuramoto_equilibrium2} 
has at most $\binom{2n-2}{n-1}$
equilibria \cite{baillieul1982}
and, for $n\geq4$, it is currently unknown
if this bound can be achieved.
The minimum number of equilibria
is easily observed to be $0$.

There are results regarding
the number of equilibria 
for the standard Kuramoto model 
that apply to the rank-one coupled Kuramoto
model as well.  
When $n = 2$, it is easy
to see that the maximum number of equilibria 
satisfying \eqref{eq:Kuramoto_equilibrium} is $2$.
By \textbf{IC1}, we have $\omega_2 = -\omega_1 \neq 0$,
so, without loss of generality, we assume $\omega_1 > 0$.
With $k = (1,1)$, one can verify:
\begin{itemize}
\item $2$ equilibria if $0 < \omega_{1} < \dfrac{1}{2}$;
\item $1$ equilibrium (of ``multiplicity 2'') if $\omega_1 = \dfrac{1}{2}$;
\item $0$ equilibria if $\omega_{1}>\dfrac{1}{2}$.
\end{itemize}

For $n=3$, the maximum number of equilibria is $6$
\cite{baillieul1982,XKL16}.
When $k = (1,1,1)$, Prop.~\ref{pro:opt1}
shows that equilibria can only occur if 
each $|\omega_\nu|\leq 1$.  By taking $\omega_3 = -\omega_1-\omega_2$ due to \textbf{IC1}, 
Figure~\ref{fig:realn3} plots the
regions having $0$, $2$, $4$, and $6$ distinct equilibria for
$\omega_{1},\omega_{2}\in [-1,1]$.
Such a plot has appeared previously, e.g., \cite{Bronski,Hiskens}.

\begin{figure}[hptb]
\begin{center}
\includegraphics[scale=0.35]{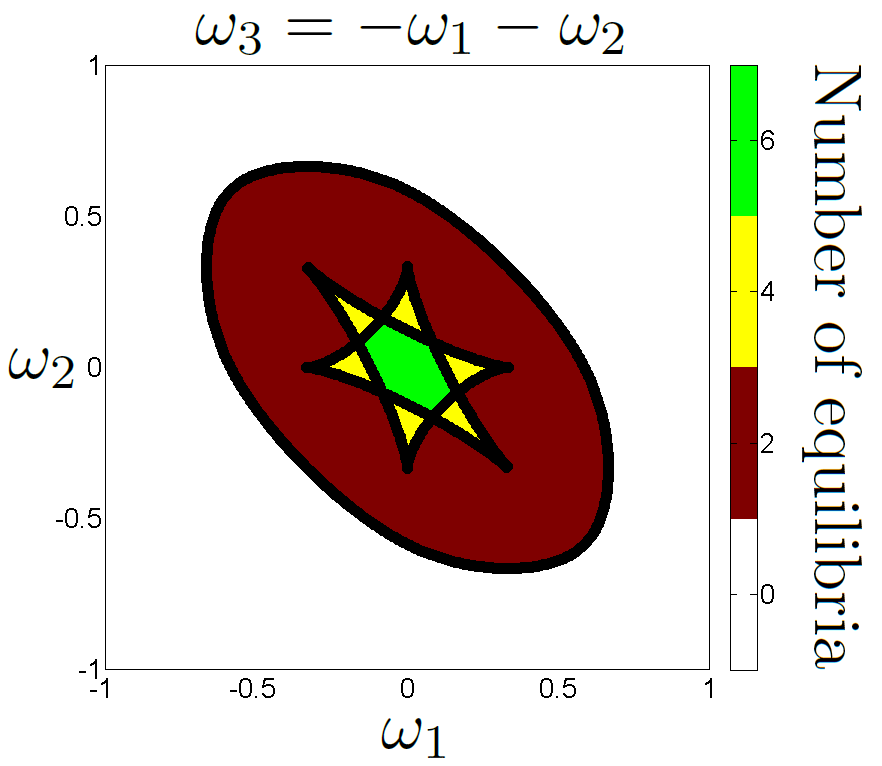}
\end{center}
\caption{Regions based on the number of equilibria 
satisfying~\eqref{eq:Kuramoto_equilibrium} when $n = 3$
and $k = (1,1,1)$
}\label{fig:realn3}
\end{figure}

For $n=4$, the maximum number of equilibria is $14$
\cite{baillieul1982,XKL16} 
and it is an open problem to determine if this bound is 
sharp.  
A recent experiment~\cite{XKL16} 
applied to the standard Kuramoto model
computed all equilibria for selected values of $\omega\in\bR^4$ 
in a relevant compact parameter space
based on a grid with step-size $1/2000$.
Since this experiment attained 
a maximum of $10$ equilibria,
they conjecture 
that the maximum number of equilibria satisfying
\eqref{eq:Kuramoto_equilibrium}
when~$n = 4$ and $k = (\sqrt{K},\sqrt{K},\sqrt{K},\sqrt{K})$ 
is~$10$, which is strictly smaller than the upper
bound of $14$.  We revisit this case 
in Ex.~\ref{ex:n4} and~\ref{ex:n42}.

\subsection{Bounding the number of equilibria}

As summarized in Section~\ref{sec:SmallCases},
the maximum number of equilibria to \eqref{eq:Kuramoto_equilibrium}
is $2,6,14$ for $n = 2,3,4$, respectively.  
Theorem~\ref{thm:UpperBound}
shows that $2^n-2$ bounds the number of equilibria
with Corollary~\ref{cor:ComplexBound} showing
that $2^n-2$ 
is actually the generic root count
for the polynomial system \eqref{eq:Kuramoto_poly}
modulo shift.

Let $\omega\in\bR^n$ and $k\in\bR_{> 0}^n$
satisfy \textbf{IC1}-\textbf{IC3}.
The following shows that the function
\begin{equation}\label{eq:UnivariatePolynomial}
g(R) ~~= \prod_{\sigma\in\{-1,+1\}^n} f_\sigma(R)~~
= \prod_{\sigma\in\{-1,+1\}^n} 
\left(
-R + \frac{1}{n}\sum_{\mu=1}^n \sigma_\mu\sqrt{k_\mu^2 R - \omega_\mu^2} \right),
\end{equation}
is actually a reducible polynomial.  

\begin{proposition}\label{prop:Degree}
The univariate function $g$ in \eqref{eq:UnivariatePolynomial} 
is a polynomial of degree $2^n$.
Moreover, 
there exists a polynomial $h(R)$ of degree $2^n-2$ with
$$g(R) = R^2\cdot h(R).$$
\end{proposition}
\begin{proof}
Since $g$ is a product over all $2^n$ conjugates,
it immediately follows that $g$ is a polynomial with leading
term $(-R)^{2^n}$ showing that $g$ is a polynomial of degree $2^n$.

In order to show that $R^2$ is a factor of $g$, we 
simply need to show that $g(0) = g'(0) = 0$.
To that end, consider $\sigma_\omega = \sign \omega \in \{-1,+1\}^n$ 
where $\sign \omega_i = 1$ if $\omega_i \geq 0$, otherwise $\sign \omega_i = -1$.
Then, 
$$f_{\sigma_\omega}(0) = \frac{1}{n}\sum_{\mu=1}^n \sign \omega_\mu \sqrt{-\omega_\mu^2} = \frac{\sqrt{-1}}{n} \sum_{\mu=1}^n \omega_\mu = 0$$
by \textbf{IC1}.  By \eqref{eq:UnivariatePolynomial}, 
this immediately shows that $g(0) = 0$ since one of the terms in the product is $0$.

By a similar argument as above, $f_{-\sigma_\omega}(0) = 0$
by \textbf{IC1}.  This shows that at least two terms in 
the product defining $g$ in \eqref{eq:UnivariatePolynomial}
are zero.  Hence, the product rule for differentiation
shows that $g'(0) = 0$.  
\end{proof}

\begin{example}
For $n = 2$, we have
$$
\hbox{\small $
g(R) = R^4 -\frac{1}{2}\left(k_1^2+k_2^2\right)R^3 + \frac{1}{16}\left(\left(k_1^2-k_2^2\right)^2 + 8\left(\omega_1^2 + \omega_2^2\right)\right)R^2
- \frac{1}{8}\left(k_1^2-k_2^2\right)\left(\omega_1^2-\omega_2^2\right)R + \frac{1}{16}\left(\omega_1^2-\omega_2^2\right)^2$}
$$
which is indeed a polynomial of degree $2^2=4$.  
Moreover, {\rm\textbf{IC1}} implies $\omega_2 = -\omega_1$ so that
\begin{equation}\label{eq:g2}
g(R) = R^2\left(R^2 -\frac{1}{2}\left(k_1^2+k_2^2\right)R + \frac{1}{16}\left(\left(k_1^2-k_2^2\right)^2 + 16\omega_1^2\right)\right).
\end{equation}
\end{example}

Proposition~\ref{prop:Degree} 
immediately provides the following upper bound.

\begin{theorem}\label{thm:UpperBound}
If $\omega\in\bR^n$ and $k\in\bR_{> 0}^n$
satisfy {\rm\textbf{IC1}--\textbf{IC3}},
then there are at most $2^n-2$ equilibria
satisfying \eqref{eq:Kuramoto_equilibrium}.
\end{theorem}
\begin{proof}
This follows from Theorem~\ref{thm: reduction}
since $g(R)$ in \eqref{eq:UnivariatePolynomial} 
has at most $2^n-2$ positive roots.
\end{proof}

\begin{corollary}\label{cor:ComplexBound}
The generic root count modulo shift 
to \eqref{eq:Kuramoto_poly} is $2^n-2$.
\end{corollary}
\begin{proof}
Reviewing the proof of Theorem~\ref{thm: reduction}
shows that $2^n-2$ also bounds the number of complex
solutions to~\eqref{eq:Kuramoto_poly}.
For $g$ in~\eqref{eq:UnivariatePolynomial},
$g''(0)\neq 0$ for generic
values of the parameters yielding
that there are generically $2^n-2$ nonzero roots
of $g$.  Hence, $2^n-2$ is the generic
root count of~\eqref{eq:Kuramoto_poly}.
\end{proof}

\begin{example}
Table~\ref{tab:comparison}
shows that 
the polynomial system \eqref{eq:Kuramoto_poly}
for $n = 4$, $\omega = (-3/4,-1/4,1/4,3/4)$, and
$k = (\sqrt{1.5},\sqrt{1.5},\sqrt{1.5},\sqrt{1.5})$
has $12$ complex roots modulo shift,
which is less than the generic root count
of $2^4-2=14$.  In fact, as in the proof
of Prop.~\ref{prop:Degree}, this is due
to the following four quantities being equal to zero:
$$\sum_{i=1}^4 \omega_i,~~~\sum_{i=1}^4 -\omega_i,~~~
\omega_1-\omega_2-\omega_3+\omega_4,~~~
-\omega_1+\omega_2+\omega_3-\omega_4.$$
Hence, $g$ in \eqref{eq:UnivariatePolynomial}
has $g(0) = g'(0) = g''(0) = g'''(0) = 0$, namely
$$g(R) =\frac{R^4}{1073741824}(64R^4-96R^3+20R^2+1)
(64R^2-24R+9)^2(64R^2-24R+1)^2.$$
\end{example}

Theorem~\ref{thm:UpperBound} provides an upper bound
of $2^n-2$ when the symmetric coupling matrix 
has rank one while~\cite{baillieul1982} provides an upper
bound of $\binom{2n-2}{n-1}$ in the general case. 
By Stirling's formula,
$$\binom{2n-2}{n-1} \approx 4^n\cdot \frac{1}{4\sqrt{\pi(n-1)}}$$
showing the bound in Theorem~\ref{thm:UpperBound} 
for the rank-one case is roughly 
the square root of the general purpose
bound from \cite{baillieul1982}.  
Due to this difference, we computed
the generic root counts for
the corresponding polynomial system
associated with~\eqref{eq:Kuramoto_equilibrium2}
when the coupling matrix $\kappa$
is a symmetric matrix of various ranks
for $n = 2,\dots,10$ using {\tt Bertini}~\cite{BHSW06}.
The results are presented in Table~\ref{table:comparisonBound}.
This data, for selected values of $r$ and $n$, 
shows that the generic root counts for
a symmetric coupling matrix of rank $r$ and rank~$r+1$
are equal whenever $n \leq 2r+1$
and differ when $n \geq 2r+2$.
In fact, the difference between 
the generic root counts for 
rank $r$ and rank~$r+1$ symmetric coupling matrices 
when $n = 2r+2$ is equal to $\binom{2r+2}{r+1} = \binom{n}{n/2}$.  We leave it as an open problem to fully 
understand the behavior for all choices of $r$ and $n$.

\begin{table}[!ht]
\begin{center}
\begin{tabular}{c|c|c|c|c|c}
$n$ & rank $1$ & rank $2$ & rank $3$ & rank $4$ & rank $5$\\
\hline
2 & 2 & 2 & 2 & 2 & 2\\
3 & 6 & 6 & 6 & 6 & 6\\
4 & 14 & 20 & 20 & 20 & 20\\
5 & 30 & 70 & 70 & 70 & 70\\
6 & 62 & 232 & 252 & 252 & 252\\
7 & 126 & 714 & 924 & 924 & 924\\
8 & 254 & 2056 & 3362 & 3432 & 3432\\
9 & 510 & 5646 & 11,860 & 12,870 & 12,870\\
10 & 1022 & 14,864 & 40,136 & 48,368 & 48,620 \\
\end{tabular}
\end{center}
\caption{Generic root counts for 
symmetric coupling matrices of various ranks}\label{table:comparisonBound}
\end{table}

\subsection{Counting equilibria for particular cases}
\label{sec:NumberReal}

Motivated by \cite{XKL16}, we use
Theorem~\ref{thm: reduction} 
to analyze the number of equilibria 
satisfying~\eqref{eq:Kuramoto_equilibrium}
for particular cases 
when $n$ is even (Theorem~\ref{thm:Even}
and Corollary~\ref{cor:Even}) and 
when $n$ is odd (Theorem~\ref{thm:Odd}).

\begin{theorem}
\label{thm:Even} Suppose that
$n\geq2$ is even and $q>0$. 
For $\omega=\left(nq,\ldots,nq,-nq,\ldots,-nq\right)$ 
and $k=(n,\ldots,n)$, there are exactly
\[
2^{n}-\sum_{-q < \ell < q }\left(
\begin{array}
[c]{c}%
n\\
n/2+\ell
\end{array}
\right)
\]
equilibria satisfying \eqref{eq:Kuramoto_equilibrium}
counting multiplicity.  Hence, the number of equilibria
changes precisely at the integers $q = 1,2,\dots,n/2$.
\end{theorem}
\begin{proof}
Since $k_\mu^2 = n^2$ and $\omega_{\mu}^{2}=n^2q^{2}$, 
Theorem~\ref{thm: reduction}
shows that we need to compute all 
$R > 0$ where
\begin{equation}\label{eq:proofR}
R = 
\frac{1}{n}\sum_{\mu=1}^{n} \sigma_{\mu}\sqrt{n^2R-n^2q^2}
= \sum_{\mu=1}^{n}\sigma_{\mu}\sqrt{R-q^{2}}=S\sqrt{R-q^2}
\end{equation}
with $S=\sum_{\mu=1}^{n}\sigma_{\mu}$
and $\sigma\in\{-1,+1\}^n$.  

If $S\leq 0$, then \eqref{eq:proofR} has no positive solutions.
Since $n$ is even, the remaining cases 
have $S\geq 2$.  Thus, the positive solutions of \eqref{eq:proofR} must satisfy
$$R  =\frac{S}{2}\left(S\pm\sqrt{S^2-4q^2}\right) > 0.$$
This yields three cases:
\begin{enumerate}
\item $2 \leq S < 2q$: \eqref{eq:proofR} has no positive solutions;
\item $S = 2q \geq 2$: \eqref{eq:proofR} has one
positive solution of multiplicity 2, namely $R = S^2/2$;
\item $S > 2q$ with $S\geq2$: \eqref{eq:proofR} has two distinct positive solutions.
\end{enumerate}

Suppose that $q$ is not an integer.  
Since $S$ is even, we have $S \neq 2q$.  
Hence, the number of equilibria is exactly
$$2\cdot\#\left\{\sigma\in\left\{+1,-1\right\}^{n} : S>2q\right\}  
= 2\cdot\#\left\{\sigma\in\left\{+1,-1\right\}^{n} : S\geq 2\left\lceil q\right\rceil \right\}
= 2\cdot\sum_{\ell=\left\lceil q\right\rceil }^{n/2}\left(
        \begin{array}
        [c]{c}%
        n\\
        n/2+\ell
        \end{array}
        \right).$$
Since $\binom{n}{n/2+\ell} = \binom{n}{n/2-\ell}$
and $2^n = \sum_{\ell=0}^n \binom{n}{\ell}$, the number of equilibria when $q$ is not an integer is
$$2\cdot\sum_{\ell=\left\lceil q\right\rceil }^{n/2}\left(
        \begin{array}
        [c]{c}%
        n\\
        n/2+\ell
        \end{array}
        \right)
= 
\sum_{\ell=-n/2}^{-\left\lceil q\right\rceil }\left(
        \begin{array}
        [c]{c}%
        n\\
        n/2+\ell
        \end{array}
        \right) + 
\sum_{\ell=\left\lceil q\right\rceil }^{n/2}\left(
        \begin{array}
        [c]{c}%
        n\\
        n/2+\ell
        \end{array}
        \right)
        = 2^{n}-\sum_{-q < \ell < q}\left(
        \begin{array}
        [c]{c}%
        n\\
        n/2+\ell
        \end{array}
        \right).
$$      

When $q$ is an integer, we need to add in the case when $S = 2q$
yielding
$$2\cdot\#\left\{\sigma\in\left\{+1,-1\right\}^{n} : S\geq 2q\right\}  = 2\cdot\sum_{\ell = q}^{n/2} \left(
        \begin{array}
        [c]{c}%
        n\\
        n/2+\ell
        \end{array}
        \right) = 
        2^{n}-\sum_{-q < \ell < q}\left(
        \begin{array}
        [c]{c}%
        n\\
        n/2+\ell
        \end{array}
        \right).
$$
\end{proof}

\begin{example}
For $n=2$ and $q>0$, the case of 
$\omega=(2q,-2q)$ and $k=(2,2)$
corresponds to $\omega=(q/2,-q/2)$ and $k = (1,1)$
as considered in Section~\ref{sec:SmallCases}.
Hence, counting multiplicity, there are two equilibria
for $q \leq 1$ and no equilibria for $q > 1$
in agreement with Theorem~\ref{thm:Even}. 
\end{example}

\begin{example}\label{ex:n4}
For $n=4$ and $q>0$, the case of 
$\omega = (4q,4q,-4q,-4q)$ and $k = (4,4,4,4)$
corresponds to $\omega=(q/4,q/4,-q/4,-q/4)$
and $k = (1,1,1,1)$ as considered in Section~\ref{sec:SmallCases}.  Figure~\ref{fig:realn4}(a)
plots the regions based on the number
of equilibria when $k = (1,1,1,1)$  
such that $\omega_3 = \omega_4 = -(\omega_1+\omega_2)/2$.
With this setup, $\omega_1 = \omega_2 = q/4$ implies 
$\omega_3 = \omega_4 = -q/4$.  Since the sign is arbitrary,
the plot in Figure~\ref{fig:realn4}(b) 
incorporates the line $\omega_1 = \omega_2 = q/4$.
By Theorem~\ref{thm:Even},
there are $10$ equilibria for $0 < |q| < 1$,
$2$ equilibria for $1 < |q| < 2$,
and no equilibria for $|q| > 2$.  
\end{example}

\begin{figure}[!hptb]
\begin{center}
\includegraphics[scale=0.32]{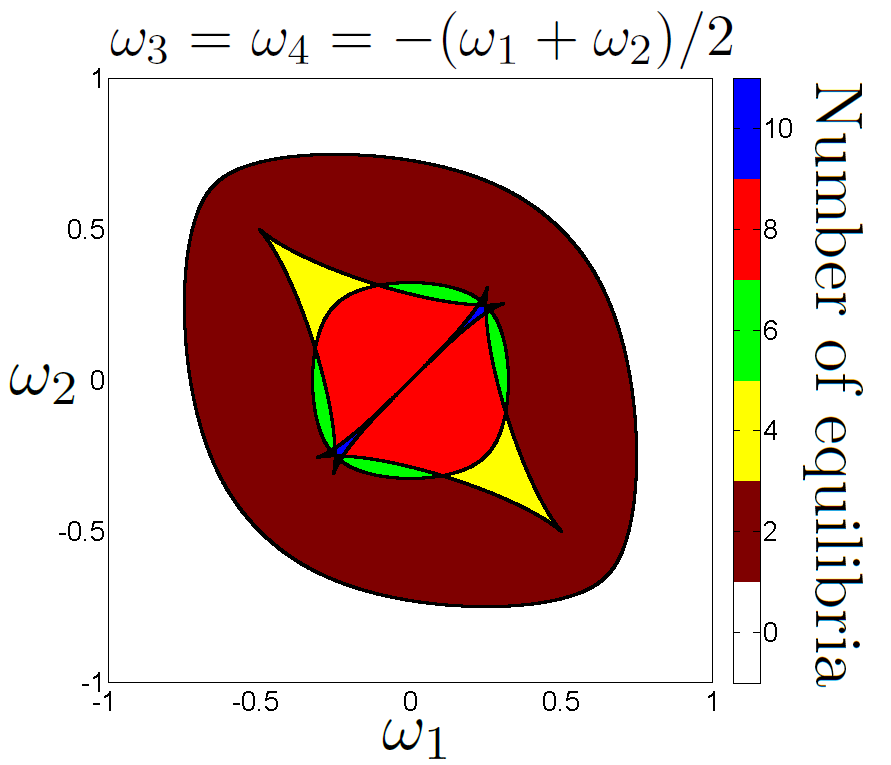}~~~~~
\includegraphics[scale=0.32]{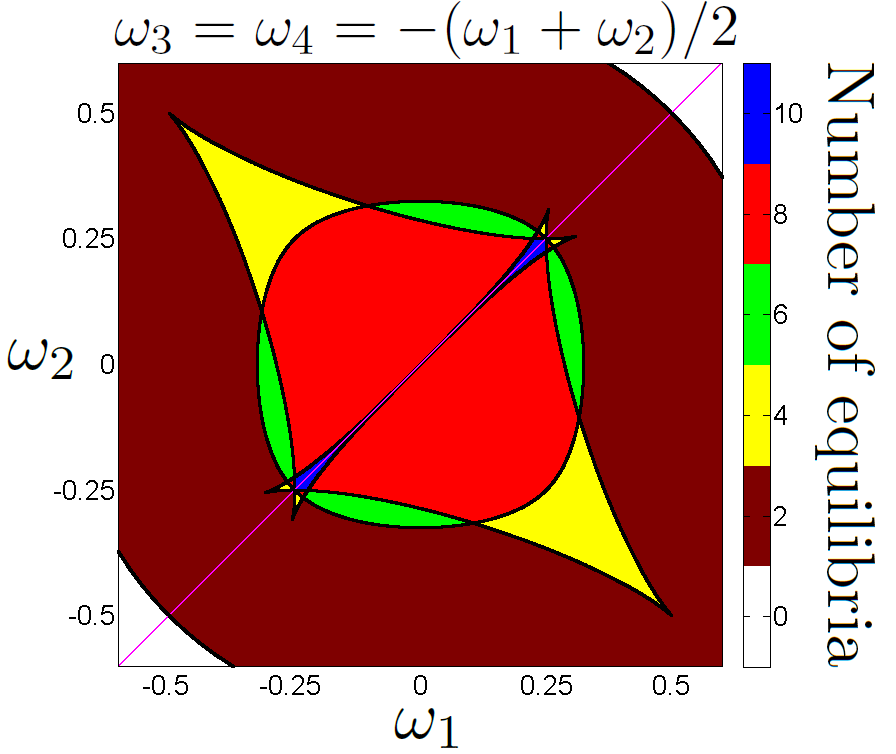}
(a)~~~~~~~~~~~~~~~~~~~~~~~~~~~~~~~~~~~~~~~~~~~~~~~~~~~~~~~~~~~~ (b) ~~~~
\end{center}
\caption{Regions based on the number of equilibria
satisfying~\eqref{eq:Kuramoto_equilibrium}
for $n = 4$ with a restricted set of~$\omega$
and $k = (1,1,1,1)$.
The diagonal line in (b) corresponds to results
from Theorem~\ref{thm:Even}.
}\label{fig:realn4}
\end{figure}

Theorem~\ref{thm:Even} immediately yields the following.

\begin{corollary}\label{cor:Even}
Suppose that $n\geq 2$ is even and $q > 0$.
The maximum number of distinct equilibria satisfying \eqref{eq:Kuramoto_equilibrium}
when $\omega = (nq,\dots,nq,-nq,\dots,-nq)$
and $k = (n,\dots,n)$ is
\begin{equation}\label{eq:CorEven}
2^n - \left(
        \begin{array}
        [c]{c}%
        n\\
        n/2
        \end{array}
        \right),
\end{equation}
which occurs for all $0 < q < 1$.
\end{corollary}

\begin{example}\label{ex:n42}
For $n=4$, Corollary~\ref{cor:Even} 
provides a maximum of $2^4-\binom{4}{2}=10$ distinct equilibria
which matches the computational results in \cite{XKL16}
as discussed in Section~\ref{sec:SmallCases}.
\end{example}

Before considering the odd case, we first define  
the constants
\begin{equation}\label{eq:Constants}
q_o = \frac{\sqrt{414 - 66\sqrt{33}}}{16} \approx
0.3690
\hbox{~~~~and~~~~}
R_o = \frac{21-3\sqrt{33}}{8} \approx 0.4708,
\end{equation}
and prove an inequality regarding them.

\begin{lemma}\label{lemma:Odd}
For $0 < q < q_o$, $R_o + \sqrt{R_o} - 2\sqrt{R_o-q^2} < 0$
where $q_o$ and $R_o$ as defined in \eqref{eq:Constants}.
\end{lemma}
\begin{proof}
Since $q < q_o$ and $R_o-q^2 > R_o -q_o^2 > 0$, we have
$$R_o + \sqrt{R_o} - 2\sqrt{R_o-q^2} <
R_o + \sqrt{R_o} - 2\sqrt{R_o-q_o^2} = 0.$$
\end{proof}

With Lemma~\ref{lemma:Odd}, we now consider the case when $n$ is odd.

\begin{theorem}\label{thm:Odd} 
Suppose that $n\geq3$ is odd and let $0 < q < q_o$
where $q_o$ is defined by \eqref{eq:Constants}.
For $\omega = (nq,\dots,nq,-nq,\dots,-nq,0)$
and $k=(n,\ldots,n)$, 
the number of equilibria satisfying \eqref{eq:Kuramoto_equilibrium} is 
        \begin{equation}  \label{eq:ThmOdd}
        2^n - \binom{n-1}{(n-1)/2}.
        \end{equation}
\end{theorem}

\begin{proof}
Since $k_\mu^2 = n^2$ for $\mu = 1,\dots,n$,
$\omega_{\nu}^{2}=n^2q^2$ for $\nu=1,\dots,n-1$, 
and $\omega_n = 0$, 
Theorem~\ref{thm: reduction}
shows that we need to compute all 
$R > 0$ with
\begin{equation}\label{eqn:odd}
R = 
\frac{1}{n}\sum_{\mu=1}^{n-1} \sigma_{\mu}\sqrt{n^2R-n^2q^2}
+ \frac{1}{n} \sigma_n \sqrt{n^2R}
= \sum_{\mu=1}^{n-1}\sigma_{\mu}\sqrt{R-q^2}
+ \sigma_n \sqrt{R} =S\sqrt{R-q^2} + \sigma_n\sqrt{R}
\end{equation}
where $S=\sum_{\mu=1}^{n-1}\sigma_{\mu}$
and $\sigma\in\{-1,+1\}^n$.  Define $p_\sigma(R) = R - \sigma_n\sqrt{R} - S\sqrt{R-q^2}$.

Since $n-1$ is even, we know that $S$ is also even.  This
yields three cases to consider.

\paragraph{$S<0$:}
Rewriting \eqref{eqn:odd} as 
        \begin{equation*}
        R - \sigma_n\sqrt{R} = S\sqrt{R- q^2}
        \end{equation*}
shows that the right-hand size is non-positive. Hence,
to have a solution, we need $\sigma_n = +1$
and $R\in(q^2,1)$.  Since $p_\sigma(q^2) = q^2-q < 0$
and $p_\sigma(1) = -S\sqrt{1-q^2} > 0$, we know
that there is at least one root in $(q^2,1)$.  
In fact, since $S \leq -2$, it is easy to see that
$p_\sigma$ is a strictly increasing function on $(q^2,1)$ since
$$p_\sigma'(R) = 1 - \frac{1}{2\sqrt{R}} + \frac{-S}{2\sqrt{R-q^2}} \geq 1 - \frac{1}{2\sqrt{R}} + \frac{1}{\sqrt{R-q^2}}
\geq 1 + \frac{1}{2\sqrt{R}} > 0$$
for all $R\in(q^2,1)$.  Thus, this case yields
one equilibrium for each $\sigma\in\{-1,+1\}^n$
such that $\sigma_n = +1$ and $S < 0$
for a total~of 
        \[
        \frac{1}{2} \left(2^{n-1}-\binom{n-1}{(n-1)/2}\right).
        \]

\paragraph{$S=0$:} 
Since \eqref{eqn:odd} becomes $R = \sigma_n \sqrt{R}$, this case requires $\sigma_n = +1$ and $R = 1$.
The total number of equilibria for this case is thus
        \[
        \binom{n-1}{(n-1)/2}.
        \]

\paragraph{$S > 0$:} We split this into two cases based on
the value of $\sigma_n$.

\subparagraph{$\sigma_n = +1$:}
Rewriting \eqref{eqn:odd} as 
        \begin{equation*}
        R - \sqrt{R} = S\sqrt{R - q^2}
        \end{equation*}
shows that the right-hand size is nonnegative.  
Hence, to have a solution, we need $R > 1$.
Since $p_\sigma(1) = -S\sqrt{1-q^2} < 0$ and 
$\lim_{R\rightarrow\infty} p_\sigma(R) = \infty$,
we know that there is at least one root in $(1,\infty)$.
In fact, the root is unique since the graph 
of $p_\sigma$ is concave up due to 
$$p_\sigma''(R) = \frac{1}{4 R^{3/2}} + \frac{S}{4 (R-q^2)^{3/2}} > 0$$
for $R > 1$.  Hence, the total number of equilibria
for this case is 
        \[
        \frac{1}{2} \left(2^{n-1}-\binom{n-1}{(n-1)/2}\right).
        \]

\subparagraph{$\sigma_n = -1$:}
We need to compute the number of roots of $p_\sigma$
for $R > q^2$.  Since $S\geq 2$ and 
$R^{3/2} > (R-q^2)^{3/2}$ for all
$R > q^2$, it follows that
$$p_\sigma''(R) = -\frac{1}{4 R^{3/2}} + \frac{S}{4 (R-q^2)^{3/2}} > 0$$
when $R > q^2$.  Hence, $p_\sigma$ is concave up
on $R > q^2$ with $p_\sigma(q^2) = q^2 + q > 0$
and $\lim_{R\rightarrow\infty} p_\sigma(R) = \infty$.
Thus, the number of roots depends on the sign of the minimum
value of $p_\sigma$ on $R > q^2$.  
Since increasing $S$ makes $p_\sigma$ more negative
and Lemma~\ref{lemma:Odd} shows
that $p_\sigma(R_o) < 0$ when $S = 2$,
there are always two roots with $R > q^2$.
Hence, the total number of equilibria for this
case is 
        \[
        2^{n-1}-\binom{n-1}{(n-1)/2}.
        \]

The result is obtained by simply 
summing the number of equilibria from all
of these cases.
\end{proof}

\begin{example}\label{ex:odd}
For $n=3$, Theorem~\ref{thm:Odd}
shows that the number of equilibria
for $\omega=(3q,-3q,0)$ and $k=(3,3,3)$ 
is $2^{3}-\binom{2}{1}=6$
whenever $0 < q < q_o$
with $q_o$ defined in \eqref{eq:Constants}.
This is equivalent to the case when
$\omega = (q/3,-q/3,0)$ and $k = (1,1,1)$
for $0 < q < q_o$.  
Since the ordering of the elements in $\omega$
is arbitrary, Figure~\ref{fig:realn3q} 
is an enhanced version
of Figure~\ref{fig:realn3} that plots, in red,
the corresponding three segments
within the region having $6$ equilibria:
\begin{itemize}
\item $\{(\alpha,0,-\alpha)~:~|\alpha|<q_0\}$ is the horizontal segment,
\item $\{(0,\alpha,-\alpha)~:~|\alpha|<q_0\}$ is the vertical segment, and
\item $\{(\alpha,-\alpha,0)~:~|\alpha|<q_0\}$ is the diagonal
segment.
\end{itemize}

\begin{figure}[!hptb]
\begin{center}
\includegraphics[scale=0.35]{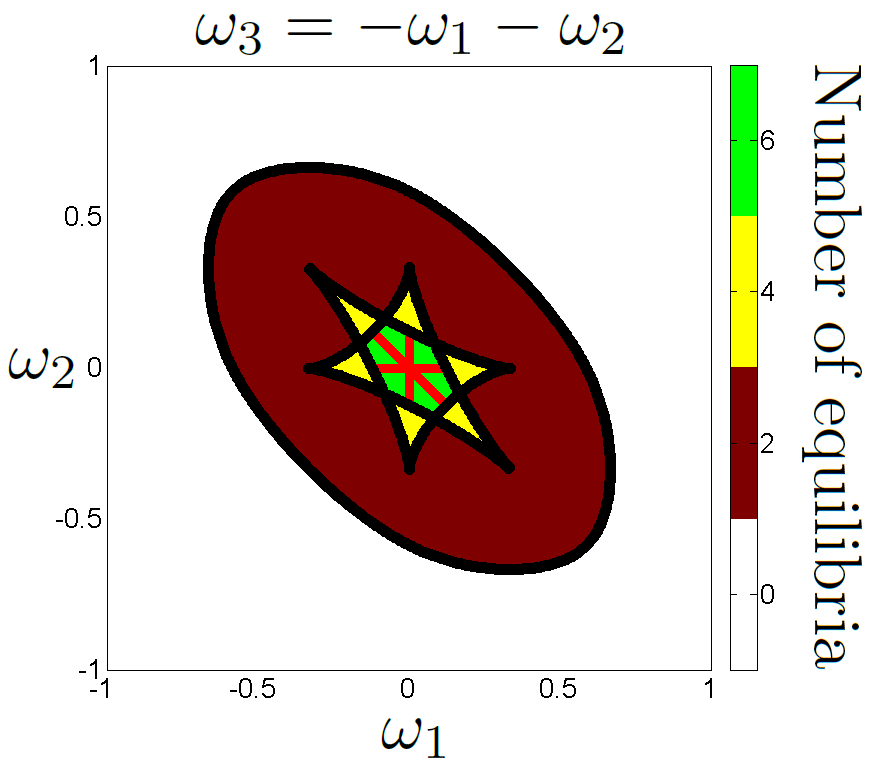}
\end{center}
\caption{Enhanced version of Figure~\ref{fig:realn3}
with the three segments from Ex.~\ref{ex:odd} plotted in red
}\label{fig:realn3q}
\end{figure}

\end{example}

The following suggests an upper bound
on the maximum number of equilibria.

\begin{conjecture}
\label{conj:MaxReal} For $n\geq2$, 
the maximum number of equilibria 
satisfying~\eqref{eq:Kuramoto_equilibrium} with $n$
oscillators~is
\[
\left\{
\begin{array}
[c]{ll}%
2^{n} - \displaystyle\binom{n}{n/2} & \hbox{~if $n$ is even,}\\
& \\
2^{n} - \displaystyle\binom{n-1}{(n-1)/2} & \hbox{~if $n$ is odd,}
\end{array}
\right.
\]
which are achieved in Corollary~\ref{cor:Even}
and Theorem~\ref{thm:Odd}, respectively.
\end{conjecture}

As summarized in Section~\ref{sec:SmallCases},
this conjecture matches the known cases
of $n = 2$ and $n = 3$, 
and agrees with the conjecture for $n = 4$ provided 
in \cite{XKL16} for the standard Kuramoto model.

\subsection{Asymptotic behavior}

\label{sec:Aspymptotic}

Even though we can only conjecture 
an upper bound on the number of equilibria,
the results from Corollary~\ref{cor:Even}
and Theorem~\ref{thm:Odd} provide
the following result: there can asymptotically be as many 
equilibria satisfying \eqref{eq:Kuramoto_equilibrium}
as the number of complex solutions to \eqref{eq:Kuramoto_poly}
modulo shift.

\begin{theorem}
\label{thm:Asymptotic} As $n\rightarrow\infty$, 
the ratio of the maximum number of equilibria 
satisfying \eqref{eq:Kuramoto_equilibrium} 
and the generic root count
to \eqref{eq:Kuramoto_poly} limits to $1$.
\end{theorem}
\begin{proof}
For each $n\geq2$, let $\Omega(n)$ denote this ratio. 
Theorems~\ref{thm:UpperBound} and~\ref{thm:Odd}
together with Corollaries~\ref{cor:ComplexBound}
and~\ref{cor:Even} show that, for every $\ell\geq1$,
$$\frac{2^{2\ell} - \binom{2\ell}{\ell}}{2^{2\ell}-2} \leq \Omega(2\ell) \leq 1 \hbox{~~~~and~~~~}
\frac{2^{2\ell+1} - \binom{2\ell}{\ell}}{2^{2\ell+1}-2} \leq \Omega(2\ell+1) \leq 1.$$

Stirling's formula yields
\[
\lim_{\ell\rightarrow\infty}\frac{\binom{2\ell}{\ell}}{2^{2\ell}-2}=\lim
_{\ell\rightarrow\infty}\frac{\frac{2^{2\ell}}{\sqrt{\pi\ell}}}{2^{2\ell}%
-2}=0
\]
so that
\[
1 \geq \lim_{\ell\rightarrow\infty}\Omega(2\ell)\geq\lim_{\ell\rightarrow\infty}%
\frac{2^{2\ell}-\binom{2\ell}{\ell}}{2^{2\ell}-2}=\lim_{\ell\rightarrow\infty
}\frac{2^{2\ell}}{2^{2\ell}-2}-\lim_{\ell\rightarrow\infty}\frac{\binom{2\ell
}{\ell}}{2^{2\ell}-2}=1-0=1.
\]

Similarly, Stirling's formula yields
\[
\lim_{\ell\rightarrow\infty}\frac{\binom{2\ell}{\ell}}{2^{2\ell+1}-2}%
=\lim_{\ell\rightarrow\infty}\frac{\frac{2^{2\ell}}{\sqrt{\pi\ell}}}%
{2^{2\ell+1}-2}=0
\]
so that
\[
1 \geq \lim_{\ell\rightarrow\infty}\Omega(2\ell+1)\geq\lim_{\ell\rightarrow\infty}%
\frac{2^{2\ell+1}-\binom{2\ell}{\ell}}{2^{2\ell+1}-2}=\lim_{\ell
\rightarrow\infty}\frac{2^{2\ell+1}}{2^{2\ell+1}-2}-\lim_{\ell\rightarrow
\infty}\frac{\binom{2\ell}{\ell}}{2^{2\ell+1}-2}=1-0=1.
\]
Therefore, $\Omega(n)\rightarrow1$ as $n\rightarrow\infty$.
\end{proof}

\section{Conclusion}

\label{sec:Conclusion}

The Kuramoto model is a standard model used
to describe the behavior of coupled oscillators
which has proven to be useful in many applications,
e.g., electrical engineering~\cite{dorfler2013, wiesenfeld1998},
biology~\cite{sompolinsky1990}, and
chemistry~\cite{bar-eli1985, kuramoto1984, neu1980}.
When the coupling matrix is a symmetric matrix of
rank one, which is a slight generalization 
of the standard Kuramoto model \eqref{eq:Kuramoto}, 
the reformulation (Theorem~\ref{thm: reduction})
permits all equilibria 
to be computed efficiently and effectively 
(Section~\ref{sec:Performance})
without the need to compute all complex solutions 
to a corresponding polynomial system.  Moreover, this 
reformulation is also useful for computing an upper bound
on the number of equilibria (Theorem~\ref{thm:UpperBound}),
computing the exact number of equilibria
for particular cases (Theorem~\ref{thm:Even}
and Theorem~\ref{thm:Odd}), and understanding
the asymptotic behavior of the maximum number
of equilibria (Theorem~\ref{thm:Asymptotic}).

Even with the broad use of the Kuramoto model
and the new results presented in this paper
regarding the equilibria, many questions still remain.
One prominent question is how to compute the maximum number of equilibria
when the coupling matrix has rank one,
which we have conjectured (Conjecture~\ref{conj:MaxReal})
is strictly smaller than the upper bound of $2^n-2$ 
for all $n \geq 4$, 
an extension of the computational results 
for the standard Kuramoto when $n = 4$ from \cite{XKL16}. It is also unknown how to extend
our decoupling method to higher rank models.
One final question regards the relationship
between the rank of the coupling matrix,
the number of oscillators, and the number of equilibria 
(Table~\ref{table:comparisonBound}),
which may yield new approaches for computing all equilibria
when the coupling matrix has rank $r > 1$.

\section*{Acknowledgment}

We would like to thank Dhagash Mehta for helpful discussions regarding the Kuramoto model,
and Bernard Lesieutre and Dan Wu for sharing a
M{\sc atlab} implementation of their 
elliptical continuation method proposed 
in~\cite{lesieutre_wu_allerton2015}.

\bibliographystyle{plain}

{\small
\begin{thebibliography}{10}

\bibitem{acebron2005}
J.A. Acebr\'on, L.L. Bonilla, C.J. {P\'erez Vicente}, F.~Ritort, and
  R.~Spigler.
\newblock The Kuramoto model: A simple paradigm for synchronization phenomena.
\newblock {\em Rev. Mod. Phys.}, 77:137--185, 2005.

\bibitem{baillieul1982}
J.~Baillieul and C.~Byrnes.
\newblock Geometric critical point analysis of lossless power system models.
\newblock {\em IEEE Trans. Circu. Syst.}, 29(11):724--737, 1982.

\bibitem{bar-eli1985}
K.~{Bar-Eli}.
\newblock {On the stability of coupled chemical oscillators}.
\newblock {\em Physica D Nonlinear Phenomena}, 14:242--252, 1985.

\bibitem{BHSW06}
D.J. Bates, J.D. Hauenstein, A.J. Sommese, and C.W. Wampler.
\newblock Bertini: Software for numerical algebraic geometry.
\newblock Available at \url{www.nd.edu/\~sommese/bertini}.

\bibitem{Bronski}
J.C. Bronski, L. DeVille, and M.J. Park.
\newblock Fully synchronous solutions and the synchronization phase transition for the finite-N Kuramoto model.
\newblock {\em Chaos}, 22:033133, 2012.

\bibitem{Buchberger1965}
B.~Buchberger.
\newblock {Ein Algorithmus zum Auffinden der Basiselemente des Restklassenringes nach einem nulldimensionalen Polynomideal [An Algorithm for Finding the Basis Elements in the Residue Class Ring Modulo a Zero Dimensional Polynomial Ideal]}.
\newblock {(Trans. in {\em Journal of Symbolic Comp., Special Issue on Logic, Math., and Comp. Science: Interactions}, 41(3-4):475--511, 2006.) }
\newblock {Mathematical Institute, University of Innsbruck, Austria, 1965.}

\bibitem{CE93}
J.~Canny and I.~Emiris.
\newblock {An Efficient Algorithm for the Sparse Mixed Resultant}.
\newblock In {\em Proc. 10th Intern. Symp. on Applied Algebra, Algebraic Algorithms, and Error-Correcting Codes, Lect. Notes in Comp. Science 263}:89--104, 1993.

\bibitem{ZC2017}
Z. Charles and A. Zachariah.
\newblock Efficiently finding all power flow solutions to tree networks. 
\newblock In {\em {55th Annu. Allerton Conf. Commun., Control, Comput.}}, Oct. 3  - Oct. 5, 2017.

\bibitem{chen2011}
H.~Chen.
\newblock {Cascaded stalling of induction motors in fault-induced delayed
  voltage recovery (FIDVR)}.
\newblock {\em MS Thesis, Univ. Wisconsin--Madison, ECE Depart.}, 2011.

\bibitem{CM15}
T.~Chen and D.~Mehta.
\newblock {On the network topology dependent solution count of the algebraic
  load flow equations}.
\newblock {\em arXiv:1512.04987}, 2015.

\bibitem{CMN16}
T.~Chen, D.~Mehta, and M.~Niemerg.
\newblock {A network topology dependent upper bound on the number of equilibria
  of the Kuramoto model}.
\newblock {\em arXiv:1603.05905}, 2016.


\bibitem{Cox2005}
D.~Cox, J.~Little, and D.~O'Shea.
\newblock {\em Using Algebraic Geometry}. 
\newblock Springer-Verlag, New York, 2005.

\bibitem{Cox2015}
D.~Cox, J.~Little, and D.~O'Shea.
\newblock {\em Ideals, Varieties, and Algorithms}. 
\newblock Springer International Publishing, 2015.

\bibitem{dorfler_bullo2012}
F.~D\"orfler and F.~Bullo.
\newblock Synchronization and transient stability in power networks and
  nonuniform kuramoto oscillators.
\newblock {\em SIAM Journal on Control and Optimization}, 50(3):1616--1642,
  2012.

\bibitem{dorfler2013}
F.~D\"orfler, M.~Chertkov, and F.~Bullo.
\newblock Synchronization in complex oscillator networks and smart grids.
\newblock {\em Proceedings of the National Academy of Sciences},
  110(6):2005--2010, 2013.





\bibitem{Emiris1999}
I.~Emiris and B.~Mourrain.
\newblock {Matrices in elimination theory}.
\newblock {\em Journal of Symbolic Comp.}, 28(1--2):3--43, 1999.

\bibitem{M2}
D.R. Grayson, and M.E. Stillman.
\newblock {Macaulay2, a software system for research in algebraic geometry}.
\newblock {Available at \url{http://www.math.uiuc.edu/Macaulay2/}}.

\bibitem{hammer1995c++}
\newblock R. Hammer, M. Hocks, U. Kulisch and D. Ratz.
\newblock {\em C++ Toolbox for Verified Computing I: Basic Numerical Problems Theory, Algorithms, and Programs},
\newblock Springer-Verlag, Berlin, 1995.

\bibitem{Regeneration}
J.D. Hauenstein, A.J. Sommese, and C.W. Wampler.
\newblock Regeneration homotopies for solving systems of polynomials.
\newblock {\em Mathematics of Computation}, 80:345--377, 2011. 

\bibitem{Hiskens}
I.A. Hiskens and R.J. Davy.
\newblock Exploring the power flow solution space boundary.
\newblock {\em IEEE Trans. Power Systems}, 16(3):389--395, 2001.

\bibitem{CXSC}
W. Kr\"amer.
\newblock {C-XSC}: A powerful environment for reliable computations in the natural and engineering sciences.
\newblock 2011 4th International Conference on Biomedical Engineering and Informatics (BMEI), Shanghai, 2011, pp. 2130--2134.
\newblock Software available at 
\url{http://www2.math.uni-wuppertal.de/~xsc/xsc/cxsc.html}.


\bibitem{kuramoto1975}
Y.~Kuramoto.
\newblock {\em Self-entrainment of a population of coupled non-linear oscillators}, 
In {\em International Symposium on Mathematical Problems in Theoretical Physics: January 23--29, 1975, Kyoto University, Kyoto/Japan},
Springer, Berlin, 1975, pp. 420--422.

\bibitem{kuramoto1984}
Y.~Kuramoto.
\newblock {\em Chemical Oscillations, Waves, and Turbulence}.
\newblock Springer, Berlin, 1984.

\bibitem{lesieutre_wu_allerton2015}
B.C. Lesieutre and D.~Wu.
\newblock {An efficient method to locate all the load flow solutions -- revisited}.
\newblock In {\em {53rd Annu. Allerton Conf. Commun., Control, Comput.}}, Sept.  29 - Oct. 2, 2015.

\bibitem{thorp1993}
W.~Ma and S.~Thorp.
\newblock {An efficient algorithm to locate all the load flow solutions}.
\newblock {\em IEEE Trans. Power Syst.}, 8(3):1077, 1993.

\bibitem{Macaulay1902}
F.~Macaulay.
\newblock {Some Formulae in elimination}
\newblock {\em Proc. London. Math. Soc.}, 33(1):3--27, 1902.

\bibitem{MNMH15}
D.~Mehta, J.D.~Hauenstein, D.K.~Molzahn, and M.~Niemerg.
\newblock {Investigating the maximum number of real solutions to the power flow
  equations: analysis of lossless four-bus systems}.
\newblock {\em arXiv:1603.05908}, 2016.

\bibitem{MMT16}
D.~Mehta, D.K. Molzahn, and K.~Turitsyn.
\newblock {Recent advances in computational methods for the power flow equations}.
\newblock In {\em American Control Conf. (ACC)}, 2016, pp. 1753--1765.

\bibitem{mehta2015algebraic}
D. Mehta, N.S. Daleo, F. D{\"o}rfler, and J.D. Hauenstein.
\newblock {Algebraic geometrization of the Kuramoto model: equilibria and
  stability analysis}.
\newblock {\em Chaos}, 25(5):053103, 2015.

\bibitem{counterexample2013}
D.K. Molzahn, B.C. Lesieutre, and H.~Chen.
\newblock {Counterexample to a continuation-based algorithm for finding all
  power flow solutions}.
\newblock {\em IEEE Trans. Power Syst.}, 28(1):564--565, 2013.


\bibitem{MMN16}
D.K. Molzahn, D.~Mehta, and M.~Niemerg.
\newblock {Toward topologically based upper bounds on the number of power flow
  solutions}.
\newblock In {\em American Control Conf. (ACC)}, 2016, pp. 5927--5932.

\bibitem{Parameter}
A.P. Morgan and A.J. Sommese.
\newblock Coefficient-parameter polynomial continuation.
\newblock {\em Appl. Math. Comput.}, 29(2):123--160, 1989.

\bibitem{neu1980}
J.C. Neu.
\newblock The method of near-identity transformations and its applications.
\newblock {\em SIAM Journal on Applied Mathematics}, 38(2):189--208, 1980.



\bibitem{salam1989}
F.M.A. Salam, L.~Ni, S.~Guo, and X.~Sun.
\newblock {Parallel processing for the load flow of power systems: the approach
  and applications}.
\newblock In {\em IEEE 28th Ann. Conf. Decis. Control (CDC)}, 1989, pp. 2173--2178.



\bibitem{sommese2005}
A.J. Sommese and C.W. Wampler.
\newblock {\em {The Numerical Solution of Systems of Polynomials Arising in
  Engineering and Science}}.
\newblock World Scientific Publishing Company, 2005.

\bibitem{sompolinsky1990}
H~Sompolinsky, D~Golomb, and D~Kleinfeld.
\newblock Global processing of visual stimuli in a neural network of coupled
  oscillators.
\newblock {\em Proceedings of the National Academy of Sciences},
  87(18):7200--7204, 1990.

\bibitem{strogatz2000}
S.H. Strogatz.
\newblock From Kuramoto to Crawford: exploring the onset of synchronization in
  populations of coupled oscillators.
\newblock {\em Physica D: Nonlinear Phenomena}, 143(1--4):1--20, 2000.

\bibitem{Sturmfels1991}
B.~Sturmfels.
\newblock {Sparse elimination theory.}
\newblock{\em Proc. Comput. Algebr. Geom. Commut. Algebra}, D.~Eisenbud, L.~Robbiano, (Eds.), 1991.




\bibitem{wiesenfeld1998}
K.~Wiesenfeld, P.~Colet, and S.H. Strogatz.
\newblock {Frequency locking in Josephson arrays: connection with the Kuramoto
  model}.
\newblock {\em Phys. Rev. E}, 57(2):1563, 1998.

\bibitem{XKL16}
X.~Xin, T.~Kikkawa, and Y.~Liu.
\newblock {Analytical solutions of equilibrium points of the standard Kuramoto
  model: 3 and 4 oscillators}.
\newblock In {\em American Control Conf. (ACC)}, 2016, pp. 2447--2452.

\end{thebibliography}

}

\end{document}